\documentclass[11pt]{article}

\usepackage[utf8]{inputenc}
\usepackage[english]{babel}
\usepackage{amsthm}
\usepackage{amsmath}
\usepackage{amsfonts}
\usepackage{amssymb}
\usepackage{hyperref}
\usepackage{theoremref}
\usepackage{graphicx}
\usepackage{float}
\usepackage{mathtools}
\usepackage{setspace}
\usepackage{siunitx}
\usepackage{geometry}
\usepackage{caption,pdfpages}
\usepackage{booktabs}
\usepackage{titlesec}
\usepackage{graphicx}
\usepackage{multirow}
\usepackage{array}
\usepackage{tikz}
\usepackage{pythonhighlight}
\usepackage{subcaption}
\usepackage{slashbox}

\newtheorem{MLtheorem}{ML Theorem}
\newtheorem{theorem}{Theorem}[section]
\newtheorem{corollary}[theorem]{Corollary}
\newtheorem{proposition}[theorem]{Proposition}
\newtheorem{conjecture}[theorem]{Conjecture}
\newtheorem{lemma}[theorem]{Lemma}
\theoremstyle{definition}
\newtheorem{definition}[theorem]{Definition}
\newtheorem{example}[theorem]{Example}
\newtheorem{remark}[theorem]{Remark}

\newcommand{\set}[2]{\{ #1 \mid #2 \}}

\title{Reinforcement Learning the Chromatic Symmetric Function }
\author{Gergely B\'erczi \thanks{Department of Mathematics, Aarhus University, email: gergely.berczi@math.au.dk} \and Jonas Kl\"uver \thanks{Department of Mathematics, Aarhus University, email: jonas.kluever@gmail.com}}
\date{}

\begin{document}

\maketitle

\begin{abstract}
    We propose a conjectural counting formula for the coefficients of the chromatic symmetric function of unit interval graphs using reinforcement learning. The formula counts specific disjoint cycle-tuples in the graphs, referred to as Eschers, which satisfy certain concatenation conditions. These conditions are identified by a reinforcement learning model and are independent of the particular unit interval graph, resulting a universal counting expression.  

\end{abstract}

\section{Introduction}

The study of proper graph colorings is a fundamental area in graph theory and theoretical computer science. For a given graph \( G \), the number of ways to color its vertices with \( n \) colors, such that no two adjacent vertices share the same color, is given by the chromatic polynomial \( \chi_G(n) \). The chromatic polynomial exhibits several intriguing properties, including log-concavity, a celebrated result established by Huh \cite{huh} The chromatic symmetric function, a multi-variable generalization of the chromatic polynomial, was introduced by Richard Stanley \cite{stanley1995}, who formulated several deep conjectures in the 1990's regarding its algebraic properties. This symmetric function is connected to various fields, including topology, statistical mechanics, representation theory and algebraic geometry, and holds a central position in algebraic combinatorics and graph theory. 

The goal of this paper is to study the chromatic symmetric function using machine learning techniques. Motivated by the Stanley-Stembridge positivity conjecture \cite{stanleystembridge,stanley1995}, our main result is a conjectured counting formula for coefficients of the elementary symmetric function expansion of the chromatic symmetric function, suggested by deep reinforcement learning models. 

Let \( G \) be a finite graph, with \( V(G) \) representing its vertices and \( E(G) \) representing its edges.

\begin{definition}
A \textbf{proper coloring} \( c \) of \( G \) is a function \( c : V \rightarrow \mathbb{N} \), where no two adjacent vertices share the same color.
\end{definition}

For a given coloring \( c \), we can associate a monomial:
\[
x_c = \prod_{v \in V} x_{c(v)},
\]
where \( x_1, x_2, \dots \) are commuting variables. Let \( \Pi(G) \) denote the set of all proper colorings of \( G \), and let $\Lambda \subset \mathbf{Z}[x_1,x_2,\ldots] $ denote the ring of symmetric functions in the infinite set of variables \( x_1, x_2, \dots \).

\begin{definition}
The \textbf{chromatic symmetric function} \( X_G \in \Lambda \) of a graph \( G \) is defined as the sum of the monomials \( x_c \) over all proper colorings \( c \) of \( G \):
\[
X_G = \sum_{c \in \Pi(G)} x_c.
\]
\end{definition}

\begin{remark}
One could consider a finite but sufficiently large number of colors \( r \), which would lead to \( X_G \in \Lambda_r \), the set of symmetric polynomials in \( r \) variables. Although this slightly complicates notation, it does not alter the results.
\end{remark}

\begin{definition}\label{def:em}
The \( m \)-th \textbf{elementary symmetric function} \( e_m \) is defined as:
\[
e_m = \sum_{i_1 < i_2 < \dots < i_m} x_{i_1} x_{i_2} \dots x_{i_m},
\]
where \( i_1, \dots, i_m \in \mathbb{N} \). Given a partition \( \lambda = (\lambda_1 \geq \lambda_2 \geq \dots \geq \lambda_k) \), we define the elementary symmetric function \( e_\lambda \) as \( e_\lambda = \prod_{i=1}^k e_{\lambda_i} \). These functions form a basis of \( \Lambda \).
\end{definition}

\begin{definition}
A symmetric function \( X \in \Lambda \) is \textbf{e-positive} if it can be written as a non-negative linear combination of elementary symmetric functions.
\end{definition}

For example, the chromatic symmetric function of \( K_n \), the complete graph on \( n \) vertices, is \( X_{K_n} = n! e_n \), which is e-positive.

\begin{definition}
The \textbf{incomparability graph} \( \text{inc}(P) \) of a poset \( P \) is the graph whose vertices are the elements of \( P \), and two vertices are adjacent if they are incomparable in \( P \).
\end{definition}

\begin{definition}
A poset \( P \) is said to be \( (a+b) \)-free if it does not contain a chain of length \( a \) and a chain of length \( b \) that are mutually incomparable.
\end{definition}

\begin{definition}
A \textbf{unit interval order} (UIO) is a poset that is isomorphic to a finite subset of \( U \subset \mathbb{R} \), where \( u \) and \( w \in U \) are incomparable if and only if \( |u - w| < 1 \).
\end{definition}

\begin{theorem}[Scott-Suppes \cite{scott}]
A finite poset \( P \) is a UIO if and only if it is \( 2+2 \)-free and \( 3+1 \)-free.
\end{theorem}

Stanley and Stembridge \cite{stanleystembridge,stanley1995} proposed the following positivity conjecture. 

\begin{conjecture}[Stanley-Stembridge \cite{stanleystembridge}]
If \( P \) is a \( 3+1 \)-free poset, then \( X_{\text{inc}(P)} \) is e-positive.
\end{conjecture}

The Stanley-Stembridge suggests that for the incomparability graph of any \( 3+1 \)-free poset \( P \), the coefficients \( c_\lambda \) in the expansion
\[
X_G = \sum_{\lambda} c_\lambda e_\lambda,
\]
are non-negative for all partitions \( \lambda \).
It has been computationally verified for posets with up to 20 elements \cite{guay}. In \cite{shareshian} Sharesahian and Wachs introduced a refined version of the chromatic symmetric functions, and conjectured that these chromatic quasi-symmetric functions are also $e$-positive. 

In 2013, Guay-Paquet \cite{guay} showed that it is enough to prove the Stanley-Stembridge conjecture for the case where the poset is both \( 3+1 \)-free and \( 2+2 \)-free, i.e., for unit interval orders. 

\begin{theorem}[Guay-Paquet \cite{guay}]
\label{thm:guay-paquet}
Let \( P \) be a \( 3+1 \)-free poset. Then, \( X_{\text{inc}(P)} \) is a convex combination of the chromatic symmetric functions
\[
\{ X_{\text{inc}(P')} \mid P' \text{ is both } 3+1 \text{-free and } 2+2 \text{-free} \}.
\]
\end{theorem}


Shortly before submitting our preprint, we became aware of Hikita's work \cite{hikita2024proof}, which presents a probabilistic interpretation of the Stanley coefficients and establishes the e-positivity conjecture. In contrast, our results offer a fundamentally different interpretation of these coefficients, discovered through the application of machine learning techniques.

\section*{Acknowledgments}

We thank Andras Szenes for introducing us to this problem. We are greatly indebted to Ádám Zsolt Wagner, whose expertise in machine learning profoundly shaped this paper. The first author was supported by DFF 40296 grant of the Danish Independent Research Fund.

\section{The main results}\label{sec:mainresults}

In this paper, instead of using supervised learning models to study the Stanley coefficients directly, we follow a more sophisticated learning strategy, which was motivated by recent work of Szenes and Rok \cite{szenesrok}. This work suggests that the Stanley coefficients should count specific cycle-tuples within the UIO. This method offers two key advantages: 
\begin{enumerate}
    \item It provides an exact counting interpretation of the coefficients. 
    \item It yields a universal formula, as the cycles we count depend only on the partition \(\lambda\), rather than on the specific UIO.
\end{enumerate}

The strategy in a nutschell is the following: given a partition \(\lambda = (\lambda_1, \dots, \lambda_s)\) the Stanley coefficient \(c_\lambda\) is given by the count of $\lambda$-Eschers in the UIO satisfying certain splitting and insertion conditions. We use a reinforcement learning agent to identify the correct conditions.
Hence the central objects are Eschers, which are cycles of a certain type in the UIO graph. Given two unit intervals $u_1$ and $u_2$ on the real line, their relative positions can fall into one of three categories:
\begin{enumerate}
    \item $u_1$ intersects $u_2$,
    \item $u_1 \cap u_2 = \emptyset$ and $u_1$ is to the left of $u_2$,
    \item $u_1 \cap u_2 = \emptyset$ and $u_1$ is to the right of $u_2$.
\end{enumerate}
Following \cite{szenesrok}, we will use the notation $u_1 \prec u_2$ for case (2), $u_1 \succ u_2$ for case (3), and $u_1 \rightarrow u_2$ if either case (1) or case (2) holds.

\begin{definition} Let $U$ be a unit interval order with intervals $u_0,\ldots, u_{n-1}$. The \textbf{area sequence of U} is the unique increasing sequence $a_U=(a_0\le \ldots \le a_{n-1})$ where 
\begin{enumerate}
    \item $0 \leq a_i \leq i$ is integer
    \item $u_i \sim u_j$ if and only if $i \geq a_j$.
\end{enumerate}
\end{definition}



\begin{remark}
The number of UIOs of length $n$ is equal to the number of increasing integer sequences satisfying condition (1), which is the Catalan number $C_n = \frac{1}{n + 1} \begin{pmatrix}
2n \\
n
\end{pmatrix}$, and the number coefficients of the chromatic symmetric function $X^U$ is the number of partitions of $n$, which we denote by $\mathcal{P}_n$. Here's a table for small number of vertices:
\begin{center}
\begin{tabular}{ |p{3cm}|p{3cm}|p{3cm}|  }
 \hline
 \# Vertices & \# UIOs, $C_n$ & \# Coeffs, $|\mathcal{P}_n|$\\
 \hline
2 & 2 & 2\\
3 & 5 & 3\\
4 & 14 & 5\\
5 & 42 & 7\\
6 & 132 & 11\\
7 & 429 & 15\\
8 & 1430 & 22\\
9 & 4862 & 30\\
20 & 6564120420 & 627\\
 \hline
\end{tabular}
\end{center}
For small UIOs it is feasible to compute all coefficients and verify their non-negativity. Computation of the Stanley coefficients is based on a combinatorial formula of Theorem \ref{thm:coefformula}, using the Stanley G-homomorphism. However, as the table suggests, the complexity grows rapidly. Calculating these coefficients becomes both memory- and time-intensive. The conjecture has been verified for UIOs of up to length 20 \cite{guay}.

\end{remark}

\begin{definition} We call a sequence $[v_0, v_1, \dots, v_{k-1}]$ of distinct elements from $U$ a \textbf{$k$-Escher} if the relation
\[
v_0 \rightarrow v_1 \rightarrow \dots \rightarrow v_{k-1} \rightarrow v_0
\]
holds. We will denote by $E_k^U$ the set of $k$-Eschers in $U$. For a pair of integers $k,l$ we denote by $E_{k,l}^U$ the set of disjoint \textbf{$(k,l)$ Escher pairs:}
\[E_{k,l}=\{(v,w) \in E_k^U \times E_l^U: v \cap w  =\emptyset\}\]
We adapt the same notation for longer \textbf{Escher-tuples}. 
\end{definition}
The operation of cyclic permutation on sequences is denoted by $\zeta$:
\[
\zeta : [v_0, v_1, \dots, v_{k-1}] \mapsto [v_1, v_2, \dots, v_{k-1},v_0].
\]
By applying powers of $\zeta$ to a $k$-Escher, we generate $k - 1$ new $k$-Eschers. Since all the resulting sequences are equivalent, we refer to this group of sequences as a "cyclic Escher" to represent the isomorphism class of sequences related by $\zeta$. An Escher, then, is a cyclic Escher with a chosen starting point. For a given escher $u$ we let $u^{v}$ denoted the escher equivalent to $u$, starting at $v \in u$.  Additionally, for a $k$-Escher, we treat the index set as integers modulo $k$, so that, for instance, $u_0 = u_k$.

Our starting point is the following observation of Szenes-Rok \cite{szenesrok}, following Stanley \cite{stanley1995}. The general formula using $G$ homomorphism will be given in the next section.

\begin{proposition}[Stanley, Szenes-Rok]\label{stanleyprop}
Let $U$ be a UIO of length $n \ge k+l$. The Stanley coefficient corresponding to $\lambda=(k,l)$ is $c_{(k,l)}^U=\# E_{k,l}^U - \# E_{k+l}^U$.
\end{proposition}

\begin{definition}
Let \( U \) be a unit interval order (UIO) of length $n$, and let \( v = [v_0, \ldots, v_{k-1}]\in E_k^U \) be a $k$-Escher. Fix a positive integer $l\le k$. If there exists an \( m \in \mathbb{N} \) such that
\[v_{m+1} \to v_{m+2} \to \dots \to v_{m+l} \to v_{m+1}
\]
and
\[
v_0 \to  \dots \to v_m \to v_{m+l+1} \to  \dots \to  v_{k-1}
\]
holds, then we call $[v_{m+1},\ldots, v_{m+l}]$ a \textbf{valid \( l \)-subescher} of \( v \). We call such $m$ a $\textbf{splitting point}$ or $\textbf{subescher starting point}$ of $v$. We denote by $\textbf{S}(v,l)$ the set of splitting points, and if this is nonempty then we let $\textbf{SEStart}(v,l)$ denote the first (smallest non-negative) splitting point (SEStart for Sub-Escher Start), and $\textbf{SEEnd}(v,l)=\textbf{SEStart}(v,l)+l$.
\end{definition}
In short, a valid subescher of length \( l \) is a subescher such that the remaining part of \( v \) still forms an Escher, illustrated as follows:

\begin{tikzpicture}[scale=.6,
mycirc/.style={circle,fill=blue!60, minimum size=0.5cm}]
\node[mycirc,label=above:{$w_m$}] (n1) at (3,0) {};
\node[mycirc,label=above:{$w_{m+1}$}] (n2) at (6,0) {};
\node[mycirc,label=above:{$w_{m+2}$}] (n3) at (9,0) {};
\node[mycirc,label=above:{$w_{m+l}$}] (n4) at (15,0) {};
\node[mycirc,label=above:{$w_{m+l+1}$}] (n5) at (18,0) {};

\node (c0) at (0,0) {$\hdots$};
\node (c2) at (12,0) {$\hdots$};
\node (c3) at (21,0) {$\hdots$};

\coordinate (z1) at (15, -2);
\coordinate (z2) at (6, -2);
\coordinate (z3) at (18, -3);
\coordinate (z4) at (3, -3);

\draw [->, very thick] (c0) -- (n1);
\draw [->, very thick] (n1) -- (n2);
\draw [->, very thick] (n2) -- (n3);
\draw [->, very thick] (n3) -- (c2);
\draw [->, very thick] (c2) -- (n4);
\draw [->, very thick] (n4) -- (n5);
\draw [->, very thick] (n5) -- (c3);
\draw [->, very thick] (n4) -- (z1) -- (z2) -- (n2);
\draw [->, very thick] (n1) -- (z4) -- (z3) -- (n5);
\end{tikzpicture}

\begin{proposition}{\cite{szenesrok},Proposition 4.5.}\label{thereissubescher}  Any $k+l$-Escher has at least one valid $k$-subescher.\end{proposition}

\begin{definition} Let $u=[u_0,\ldots, u_{k-1}] \in E_k^U$ and $v=[v_0,\ldots, v_{l-1}] \in E_l^U$ be disjoint subeschers of length $k,l$ respectively, that is, $(u,v)\in E_{k,l}^U$ is a $(k,l)$ Escher-pair. We call 
$i \in \mathbb{Z}$ an $\textbf{insertion point}$ for $(u,v)$ if $u_i \rightarrow v_{i+1}$ and $v_i \rightarrow u_{i+1}$ holds. In this case we can concatenate $u$ and $v$ to get a length $k+l$ escher $u+_iv$ as follows:
\begin{enumerate}
\item for $i = 0, ..., k-1$ 
$$u+_iv = [u_0, ..., u_i, v_{i+1}, ..., v_{i+l}, u_{i+1}, ..., u_{k-1}]$$
\item for $i \geq k$ we define
$$u+_iv = [u_q, ..., u_i, v_{i+1}, ..., v_{i+l}, u_{i+1}, ..., u_{q+k-1}]$$ 
where $q \in (i-k,i]$ is choosen s.t $k|q$. 
\end{enumerate}
Denote the set of insertion points by $\textbf{I}(u,v)$, and if this is nonempty, then the first (smallest non-negative) insertion point by $\textbf{FirstIns}(u,v)$. 
\end{definition}


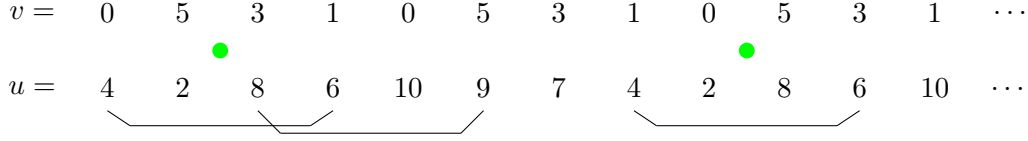
\begin{figure}
\begin{center}
    \begin{tikzpicture}
        \node at (-1, 0) {$v=$};
        \node at (-1, -1) {$u=$};
        
        \node at (0, 0) {0};
        \node at (1, 0) {5};
        \node at (2, 0) {3};
        \node at (3, 0) {1};
        \node at (4, 0) {0};
        \node at (5, 0) {5};
        \node at (6, 0) {3};
        \node at (7, 0) {1};
        \node at (8, 0) {0};
        \node at (9, 0) {5};
        \node at (10, 0) {3};
        \node at (11, 0) {1};
        \node at (12, 0) {$\ldots$};

        \node at (0, -1) {4};
        \node at (1, -1) {2};
        \node at (2, -1) {8};
        \node at (3, -1) {6};
        \node at (4, -1) {10};
        \node at (5, -1) {9};
        \node at (6, -1) {7};
        \node at (7, -1) {4};
        \node at (8, -1) {2};
        \node at (9, -1) {8};
        \node at (10, -1) {6};
        \node at (11, -1) {10};
        \node at (12, -1) {\ldots};
        
        \fill[green] (1.5, -0.5) circle (3pt);
        \fill[green] (8.5, -0.5) circle (3pt);

        \draw (0, -1.3) -- (0.3, -1.5);
        \draw (0.3, -1.5) -- (2.7, -1.5);
        \draw (2.7, -1.5) -- (3, -1.3);

        \draw (2, -1.3) -- (2.3, -1.6);
        \draw (2.3, -1.6) -- (4.7, -1.6);
        \draw (4.7, -1.6) -- (5, -1.3);

        \draw (7, -1.3) -- (7.3, -1.5);
        \draw (7.3, -1.5) -- (9.7, -1.5);
        \draw (9.7, -1.5) -- (10, -1.3);
    \end{tikzpicture}
    \caption{Insertion points and sub-Eschers for the Escher pair $u=[4,2,8,6,10,9,7], v=[0,5,3,1]$ in the UIO $U=(0,0,1,1,2,3,3,4,6,7,9)$. Green dots indicate insertion points (i.e $2 \to 3$ and $5 \to 8$ in $U$, and underbrace indicates length $4$ sub-Eschers in $u$.}
    \label{fig:subeschers}
\end{center}
\end{figure}

\begin{example}
Let $U$ be the UIO with $11$ intervals given by the area sequence $U=(0,0,1,1,2,3,3,4,6,7,9)$ as in Figure \ref{fig:subeschers}. Let 
$$u=[4,2,8,6,10,9,7], v=[0,5,3,1]$$ 
forming a $(7,4)$ Escher-pair $(u,v) \in E_{4,3}^U$.  We put $(u,v)$ periodically under each other with a period $lcm(7,4)=28$ as in Figure \ref{fig:subeschers}. We indicated by green array the insertion points. 
\end{example}

In contrast to Proposition \ref{thereissubescher}, there are Escher pairs that have no insertion points. Hence Proposition \ref{stanleyprop} intuitively suggests that for a fixed pair $(k,l)$ the Stanley coefficient $c_{(k,l)}^U$ counts $(k,l)$ Escher-pairs in $U$ which cannot be concatenated to an $(k+l)$-Escher. The technical difficulty lies in handling Eschers as linear, not cyclic objects. 

To follow this intuition, to an Escher pair $(u,v) \in E_{k,l}^U$ we associate a 3-dimensional rational vector, which we call the \textbf{core vector}, defined as the function

\[
\tau^2: E_{k,l}^U \to \mathbb{Q}^4
\]
\[
\tau^2(u,v) = 
\begin{cases}
    (0,-1,-1,-0.5) & \text{if } \mathbf{I}(u,v) = \emptyset, \\
    (0,\mathbf{SEStart}(u,l), \mathbf{SEEnd}(u,l), \mathbf{FirstIns}(u,v) + 0.5) & \text{if } \mathbf{I}(u,v) \neq \emptyset \ \& \ 0 < \mathbf{SEStart}(u,l), \\
    (0,-1,-1, \mathbf{FirstIns}(u,v) + 1.5) & \text{if } \mathbf{I}(u,v) \neq \emptyset \ \& \ 0 = \mathbf{SEStart}(u,l).
\end{cases}
\]

\vspace{0.3pt}

By Proposition \ref{thereissubescher} all cases are covered. We call $\tau^2$ the \textbf{core representation} for Escher pairs. We spent considerable time to come up with the right core coordinates, and generalise the core representation to Escher tuples of length 3 and higher. We see that the second and third coordinates are integers, the fourth coordinate is always half-integer, this is crucial in the RL model. We will denote the $i$th coordinate with $\tau^2_i(u,v)$, and simply refer to them as $0,\textbf{SEStart,SEEnd,FirstIns}$, but keeping in mind their more sophisticated definition. 

The intuition is that the numerical relationship among the core coordinates characterizes those $(k,l)$ Escher pairs which do not concatenate to a $(k+l)$ Escher in Proposition \ref{stanleyprop}. Our first result supports this intuition. 
\begin{MLtheorem}[\textbf{Counting formula for pairs}]\label{mainthm1} For any UIO of length at least $k+l$ 
\begin{equation}\label{condpairs}
    c_{(k,l)}^U \ge \# \{(u,v) \in E_{k,l}^U: \tau^2_3(u,v)<\tau^2_4(u,v)\}
\end{equation}
and equality holds for all UIOs when $\lambda=(k,k)$. Moreover, equality holds for almost all $UIO$s even when $k\neq l$: for a fixed $(k,l)$ the mismatch ratio is less than $0.5\%$, see Remark \ref{remark:expl} for explanation. 
\end{MLtheorem}

\begin{remark}\label{remark:expl}
    \begin{enumerate}
        \item Table \ref{tab:pairs} illustrates Theorem \ref{mainthm1}. We observe that for pairs $(k,l)\neq (5,3),(5,4)$ we get perfect matching and equality in \eqref{condpairs}. For $\lambda = (5, 3)$, there is a mismatch in \eqref{condpairs} for only 1 out of 429 UIOs. This exceptional UIO is 
    \[U=(0,0,1,2,3,4,5,6)\]
    which does not exhibit any obvious unusual properties, making it unclear why the mismatch occurs specifically with this UIO.
    \item More notably, from Table \ref{tab:pairs} the sum of the 429 Stanley coefficients is 52,500, while the total absolute error is just 1. 

    \item This phenomenon highlights why proving a counting formula is so challenging: we encounter extremely rare exceptional UIOs with minimal mismatches, yet there is no apparent explanation for their occurrence.
\end{enumerate}
\end{remark}

We can illustrate the conditions in \eqref{condpairs} as a graph, whose vertices are the cores of the representation, and an oriented edge from core1 to core2 means that $core1<core2$, see Figure \ref{fig:condgraphpair}. 

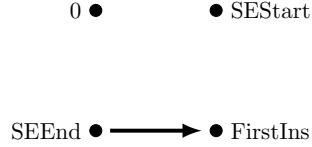
\begin{figure}[ht]
    \centering
    \begin{tikzpicture}[node distance=1.5cm, auto, scale=0.8, every node/.style={transform shape}]

        \node[draw, circle, fill=black, inner sep=2pt, label=left:{0}] (A) at (0,0) {};
        \node[draw, circle, fill=black, inner sep=2pt, label=right:{SEStart}] (B) at (2,0) {};
        \node[draw, circle, fill=black, inner sep=2pt, label=left:{SEEnd}] (C) at (0,-2) {};
        \node[draw, circle, fill=black, inner sep=2pt, label=right:{FirstIns}] (D) at (2,-2) {};

        \draw[->, >=latex, line width=1.5pt, shorten >=3pt, shorten <=3pt] (C) -- (D);
    \end{tikzpicture}
    \caption{Condition graph for Escher pairs}
    \label{fig:condgraphpair}
\end{figure}

For a triple $\lambda = (k, l, m)$, the idea is that the Stanley coefficient $c_{(k, l, m)}^U$ should count the $(k, l, m)$ Escher tuples where no two tuples concatenate at the deepest level—in other words, any two tuples remain non-concatenated. For a triple $(u, v, w) \in E_{k, l, m}^U$, we use the core coordinates for the pairs $(u, v)$, $(u, w)$, and $(v, w)$. However, if $u$ and $v$ do concatenate, then we need to store the core coordinates for the pair $(uv, w)$, $(uw, v)$ and $(vw, u)$ as well. This results in a total of 19 core coordinates as follows:  
\[\tau^3: E_{k,l,m}^U \to \mathbb{Q}^{19}\]
\[
    \tau^3(u,v,w)=(0,\tau^2(u,v)[1:],\tau^2(v,w)[1:],\tau^2(u,w)[1:],
    \tilde{\tau}^2(uv,w),\tilde{\tau}^2(uw,v),\tilde{\tau}^2(vw,u))
\]
where 
\begin{enumerate}
    \item $\tau^2(u,v)[1:]$ stands for the last 3 coordinates of $\tau^2(u,v)$, that is, we drop the $0$ from the beginning.
    \item the definition of $\tilde{\tau}^2$ is quite subtle due to the cases where $u$ and $v$ do not concatenate, making the cores involving $uv$ undefined or irrelevant. We let
\begin{equation}\label{def:tildetau}\tilde{\tau}^2(uv,w)=\begin{cases}
    (-1,-1,-0.5) & \text{ if } \mathbf{I}(u,v)=\emptyset \\
    \tau^2(u+_iv,w)[1:] & \text{ if } \mathbf{FirstIns}(u,v)=i
\end{cases}
\end{equation}
\end{enumerate}

For a more concise and descriptive reference we use the following shorthand notation: for a triple $(u,v,w)\in E_{(k,l,m)}^U$ we put
\[\mathbf{SEStart}(u,v)=\mathbf{SEStart}(u,|v|) \text{ and } \mathbf{SEEnd}(u,v)=\mathbf{SEEnd}(u,|v|).\]
So the notations for the $19$ core coordinates are 
\begin{align*}
    \tau_3(u,v,w)=& (0,
    \mathbf{SEStart}(u,v), \mathbf{SEEnd}(u,v), \mathbf{FirstIns}(u,v), \\
    & \mathbf{SEStart}(v,w), \mathbf{SEEnd}(v,w), \mathbf{FirstIns}(v,w), \\
    & \mathbf{SEStart}(u,w), \mathbf{SEEnd}(u,w), \mathbf{FirstIns}(u,w), \\
    & \mathbf{SEStart}(uv,w), \mathbf{SEEnd}(uv,w), \mathbf{FirstIns}(uv,w), \\
    & \mathbf{SEStart}(uw,v), \mathbf{SEEnd}(uw,v), \mathbf{FirstIns}(uw,v), \\
    & \mathbf{SEStart}(vw,u), \mathbf{SEEnd}(vw,u), \mathbf{FirstIns}(vw,u))
\end{align*}

The condition graph we are looking for has maximum 19 vertices, and our RL agent found the following
\begin{MLtheorem}[\textbf{Counting formula for triples}]
    \begin{enumerate}
    \item For a triple $(k,l,m)$ with $k \ge l \ge m$
    \[
    c_{(k,l,m)}^U \approx \# \left\{ (u,v,w) \in E_{k,l,l}^U \ : \ 
    \begin{aligned}
        & \mathbf{SEEnd}(u,v) < \mathbf{FirstIns}(u,v), \\
        & \mathbf{SEEnd}(u,w) < \mathbf{FirstIns}(u,w), \\
        & \mathbf{SEEnd}(v,w) < \mathbf{FirstIns}(v,w)
    \end{aligned}
    \right\}
    \]
    where $\approx$ means that the mismatch ratio is $<3\%$ and total absolute error ratio less than 1\% (see Table \ref{tab:triple1}).
    \item For $\lambda=(k,l,l)$ with $k\ge l$ we have equality above, and our counting formula gives the Stanley coefficient (see Table \ref{tab:triple1})
    \item There is a modified canonical model which gives lower bound for all triples $(k,l,m)$ with $k \ge l \ge m$ (see Table \ref{tab:triple2}):
    \[
    c_{(k,l,m)}^U \ge \# \left\{ (u,v,w) \in E_{k,l,m}^U \ : \ 
    \begin{aligned}
        & 0 \ge \mathbf{SEStart}(u,v), 0 \ge \mathbf{SEStart}(u,w), 0 \ge \mathbf{SEStart}(v,w) \\
        & \mathbf{SEEnd}(u,v) < \mathbf{FirstIns}(u,v), \\
        & \mathbf{SEEnd}(u,w) < \mathbf{FirstIns}(u,w), \\
        & \mathbf{SEEnd}(v,w) < \mathbf{FirstIns}(v,w)
    \end{aligned}
    \right\}
    \]
    \end{enumerate}
\end{MLtheorem}
The corresponding condition graphs are shown in Figure \ref{fig:condgraph3}.

\begin{figure}[ht]
    \begin{subfigure}[b]{0.5\textwidth}
    \centering
    \begin{tikzpicture}[node distance=2cm, auto, scale=0.7, every node/.style={transform shape}]
        
        \node[draw, circle, fill=black, inner sep=2pt, label=above:{SEEnd(u,v)}] (A) at (90:3) {};
        \node[draw, circle, fill=black, inner sep=2pt, label=above right:{\text{FirstIns}(u,v)}] (D) at (30:3) {};
        \node[draw, circle, fill=black, inner sep=2pt, label=right:{SEEnd(u,w)}] (B) at (-30:3) {};
        \node[draw, circle, fill=black, inner sep=2pt, label=below right:{\text{FirstIns}(u,w)}] (E) at (-90:3) {};
        \node[draw, circle, fill=black, inner sep=2pt, label=below left:{SEEnd(v,w)}] (C) at (-150:3) {};
        \node[draw, circle, fill=black, inner sep=2pt, label=left:{\text{FirstIns}(v,w)}] (F) at (150:3) {};

        \draw[->, >=latex, thick, shorten >=1mm, shorten <=1mm, line width=1.2pt] (A) -- (D);
        \draw[->, >=latex, thick, shorten >=1mm, shorten <=1mm, line width=1.2pt] (B) -- (E);
        \draw[->, >=latex, thick, shorten >=1mm, shorten <=1mm, line width=1.2pt] (C) -- (F);
    \end{tikzpicture}
    \caption{Best approximating graph and perfect match for $(k,l,l)$}
    \end{subfigure}
    \hfill
    \begin{subfigure}[b]{0.5\textwidth}
    \centering
    \begin{tikzpicture}[node distance=2cm, auto, scale=0.75, every node/.style={transform shape}]

        \node[draw, circle, fill=black, inner sep=2pt, label=above:{SEEnd(u,v)}] (A) at (90:3) {};
        \node[draw, circle, fill=black, inner sep=2pt, label=above right:{\text{FirstIns}(u,v)}] (D) at (54:3) {};
        \node[draw, circle, fill=black, inner sep=2pt, label=right:{SEEnd(u,w)}] (B) at (18:3) {};
        \node[draw, circle, fill=black, inner sep=2pt, label=below right:{\text{FirstIns}(u,w)}] (E) at (-18:3) {};
        \node[draw, circle, fill=black, inner sep=2pt, label=below:{SEEnd(v,w)}] (C) at (-54:3) {};
        \node[draw, circle, fill=black, inner sep=2pt, label=below left:{\text{FirstIns}(v,w)}] (F) at (-90:3) {};
        \node[draw, circle, fill=black, inner sep=2pt, label=below left:{SEStart(v,w)}] (G) at (-126:3) {};
        \node[draw, circle, fill=black, inner sep=2pt, label=left:{SEStart(u,v)}] (H) at (-162:3) {};
        \node[draw, circle, fill=black, inner sep=2pt, label=above left:{SEStart(u,w)}] (I) at (126:3) {};
        \node[draw, circle, fill=black, inner sep=2pt, label=above:{0}] (J) at (162:3) {};

        \draw[->, >=latex, thick, shorten >=1mm, shorten <=1mm, line width=1.2pt] (A) -- (D);
        \draw[->, >=latex, thick, shorten >=1mm, shorten <=1mm, line width=1.2pt] (B) -- (E);
        \draw[->, >=latex, thick, shorten >=1mm, shorten <=1mm, line width=1.2pt] (C) -- (F);
        
        \draw[->, >=latex, thick, shorten >=1mm, shorten <=1mm, line width=1.2pt] (G) -- (J);
        \draw[->, >=latex, thick, shorten >=1mm, shorten <=1mm, line width=1.2pt] (H) -- (J);
        \draw[->, >=latex, thick, shorten >=1mm, shorten <=1mm, line width=1.2pt] (I) -- (J);

    \end{tikzpicture}
    \caption{Lower bound for all triples}
    \end{subfigure}
    \caption{The condition graphs for Escher triples}
    \label{fig:condgraph3}
\end{figure}
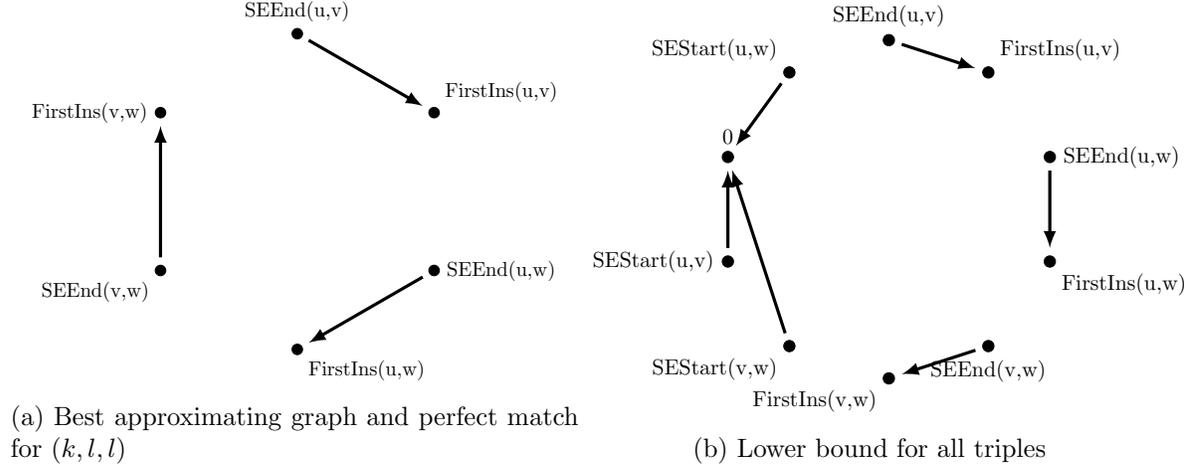

We concluded similar results for quadruples, see Table \ref{tab:quadruple1},\ref{tab:quadruple2} and we summarize our findings as 

\begin{MLtheorem}[\textbf{General counting formula}]\label{mainthm3}  Let $\lambda=(\lambda_1 \ge  \lambda_2 \ge \ldots \ge \lambda_r)$ be a partition then 
\begin{enumerate}
\item \[
    c_{\lambda}^U \approx \# \left\{ (u_1,\ldots u_r) \in E_{\lambda}^U \ : \  \mathbf{SEEnd}(u_i,u_j) < \mathbf{FirstIns}(u_i,u_j) \text{ for all } 1\le i<j \le r
    \right\}
    \]
    where $\approx$ means that the mismatch ratio is $<1\%$ for most $\lambda$'s 
    \item For $\lambda=(k,l,l,\ldots l)$ with $k\ge l$ we have equality above, and our counting formula gives the Stanley coefficient. 
    \item Slight modification provides lower bound for the coefficients:
    \[
    c_{\lambda}^U \ge \# \left\{ (u_1,\ldots u_r) \in E_{\lambda}^U \ : \  
    \begin{aligned}
        & \mathbf{SEEnd}(u_i,u_j) < \mathbf{FirstIns}(u_i,u_j) \text{ for all } 1\le i<j \le r \\ 
        & 0 \ge \mathbf{SEStart}(u_i,u_j) \text{ for all } 1\le i<j \le r 
    \end{aligned}    
    \right\}
    \]
    \end{enumerate}

\end{MLtheorem}

\section{Mathematical background}

\subsection{Stanley's G-homomorphism} 

Let $\Lambda \subset \mathbb{Z}[x_1,x_2,\ldots]$ denote the ring of symmetric functions in the variables $x_1,x_2,\ldots$, generated as an algebra by the the elementary symmetric functions $\{e_i: i=0,1,\ldots\}$ defined in Definition \ref{def:em}, and as a $\mathbb{Z}$-module by $\{e_\lambda\}$.
Given a finite graph $G$ with vertices $V=\{v_1, ..., v_n\}$, let $\Lambda_G = \mathbb{Z}[v_1, ..., v_n]$, be the polynomial ring on the vertices. Stanley's G-homomorphism $\rho_G:\Lambda \to \Lambda_G$ is defined as the unique ring-homomorphism s.t.
$$\rho_G(e_i) = \sum_{\substack{ S \subseteq V\\ \text{S is empty}, |S|=i }} \prod\limits_{v \in S} v$$
where the sum is taken over all subgraphs on $i$ vertices without edges. For $f \in \Lambda$ we set $f^{G} = \rho_G(f)$. 
\begin{example}
Let $G$ be a finite graph on $n$ vertices. Then $e_1^G = \sum_{v \in G} v$, if furthermore $G$ has at least 1 edge, then $e_n^{G} = 0$. For a partition $\lambda = (\lambda_1 \geq \lambda_2 \geq \dots \geq \lambda_r)$ we have
\[
e^G_\lambda = \prod_{i=1}^{r} e^G_i,
\]
\end{example}

\begin{definition} Let $\alpha:V \to \mathbb{N}$ be a function assigning a nonnegative integer to each vertex of $G$.
\begin{enumerate}
\item   We use the shorthand notation $\bold{v}^{\alpha} = \prod\limits_{v \in V} v^{\alpha(v)} \in \Lambda_G$, and for $f \in \Lambda$, let $[f^G, \bold{v}^{\alpha}]$ denote the coefficient of the monomial $\bold{v}^{\alpha}$ in $f^{G}$. 
\item We denote by $G^{\alpha}$, the graph where each vertex $v_i$ is replaced by the complete graph $K_{\alpha(i)}$, and for $i \neq j$ a vertex from $K_{\alpha(i)}$ is connected to a vertex from $K_{\alpha(j)}$ if and only if $v_i$ and $v_j$ are connected by an edge in $G$.
\end{enumerate}
\end{definition}    

Note that the function $f \mapsto [f^G, \bold{v}^{\alpha}]$ is linear, but not a ring-homomorphism: given a graph $G$ with at least one edge we have $[e_i^{G}, \bold{v}^{\alpha}] = 0$ for all $i \in \mathbb{N}$ but $[((e_1)^n)^G, \bold{v}^{\alpha}] = n!$.

\begin{definition} We have other canonical bases of symmetric functions:
\begin{enumerate}
\item Denote by $p_m$ the $m$-th power sum symmetric function 
$p_m = \sum_{i \in \mathbb{N}} x_i^m$.
Given a partition $\lambda = (\lambda_1 \geq \lambda_2 \geq \dots \geq \lambda_k)$, we define the \textbf{power sum symmetric function} 
\[
p_\lambda = \prod_{i=1}^k p_{\lambda_i}.
\]
\item Given a partition $\lambda = (\lambda_1 \geq \lambda_2 \geq \dots \geq \lambda_k)$, we define the \textbf{monomial symmetric function}
\[
m_\lambda = \sum_{i_1 < i_2 < \dots < i_k} \sum_{\sigma \in S_k(\lambda)} x_{i_1}^{\lambda_1} x_{i_2}^{\lambda_2} \dots x_{i_k}^{\lambda_k},
\]
where the inner sum is taken over the set of all permutations of the sequence $\lambda$, denoted by $S_k(\lambda)$.
\item For a partition $\lambda = (\lambda_1 \geq \lambda_2 \geq \dots \geq \lambda_k)$, define the \textbf{Schur functions} as
\[
s_\lambda = \det(e_{\lambda^*_i + j - i})_{i,j},
\]
where $\lambda^*$ is the conjugate partition to $\lambda$.
\end{enumerate}
The sets $\{p_\lambda\}$,$\{m_\lambda\}$,$\{s_\lambda\}$ form three basis of the ring $\Lambda$ of symmetric functions.
\end{definition} 

\begin{proposition}[Stanley \cite{ghomo}]\thlabel{Tcoef}
Let $\bold{x}=(x_1,x_2,\ldots )$ denote an infinite set of variables and $\bold{v}=(v_1,\ldots, v_n)$ the variables given by the vertices of $G$. Let $$T(\bold{x}, \bold{v}) = \sum_{\lambda}e_{\lambda}(\bold{x})m_{\lambda}^{G}(\bold{v})$$
then 
$$[T(\bold{x}, \bold{v}), \bold{v}^{\alpha}]\prod\limits_{v \in V}\alpha(v)! = X_{G^{\alpha}}$$
\end{proposition}

Using \thref{Tcoef}, Stanley proved \cite{ghomo} that the simultaneous e-positivity of a large class of graphs is equivalent to monomial positivity of $m_{\lambda}^G$ for all partitions:

\begin{theorem}(Stanley) For every finite graph $G$ we have
\begin{enumerate}
\item $X_{G^{\alpha}}$ is s-positive for every $\alpha:V(G) \to \mathbb{N} \iff s_{\lambda}^G \in \mathbb{N}[V(G)]$ for every partition $\lambda$
\item $X_{G^{\alpha}}$ is e-positive for every $\alpha:V(G) \to \mathbb{N} \iff m_{\lambda}^G \in \mathbb{N}[V(G)]$ for every partition $\lambda$
\end{enumerate}
\end{theorem}

\begin{proof}
Let $\bold{x}$, $\bold{y}$ denote 2 infinite sets of independent variables. The Cauchy product 
$$C(\bold{x}, \bold{y}) = \prod\limits_{i,j}(1+x_iy_j)$$
yields the identities \cite[(4.2'), (4.3')]{macdonald} 
$$\sum_{\lambda}s_{\lambda}(\bold{x})s_{\lambda^*}(\bold{y}) = \sum_{\lambda}m_{\lambda}(\bold{x})e_{\lambda}(\bold{y}) = \sum_{\lambda}e_{\lambda}(\bold{x})m_{\lambda}(\bold{y})$$
with the sums over all partitions of any length. By viewing only $\bold{y}$ as the indeterminates and $\bold{x}$ as constant we can apply the G-homomorphism and get
$$\sum_{\lambda}s_{\lambda}(\bold{x})s_{\lambda^*}^{G}(\bold{v}) = \sum_{\lambda}m_{\lambda}(\bold{x})e_{\lambda}^{G}(\bold{v}) = \sum_{\lambda}e_{\lambda}(\bold{x})m_{\lambda}^{G}(\bold{v})$$
Using \thref{Tcoef} gives
$$X_{G^{\alpha}} = (\prod\limits_{v \in V}\alpha(v)!)\sum\limits_{\lambda}e_{\lambda}(\bold{x})[m_{\lambda}^{G}(\bold{v}), \bold{v}^{\alpha}] = (\prod\limits_{v \in V}\alpha(v)!)\sum\limits_{\lambda}s_{\lambda}(\bold{x})[s_{\lambda^*}^{G}(\bold{v}), \bold{v}^{\alpha}]$$
\end{proof}

If the chromatic polynomial is
$X^{G} = \sum\limits_{\lambda}c_{\lambda}^Ge_{\lambda}$, then using \thref{Tcoef} with $\alpha \equiv 1$, we get \[c_{\lambda}^G = [m_{\lambda}^{G}, v_1 \cdot v_2 \cdot ... \cdot v_n]\]
for any partition $\lambda$. In particular, the Stanley conjecture follows if 
\[[m_{\lambda}^{U}, v_1 \cdot v_2 \cdot ... \cdot v_n] \geq 0\]
for all unit interval order graphs $U$.

\subsection{Eschers} 
Let $U$ be a unit interval graph on $N$ vertices. Recall the notation $E_\lambda^U=\{\lambda \text{-Escher tuples in U}\}$, and $X^U=\sum_{\lambda} c_\lambda^U e_\lambda$ is the Stanley chromatic function. Let us work out the coefficient $c_N^U$ first. 
\begin{definition} Let $U$ be a UIO. We call a sequence $w = [w_1, \dots, w_k]$ of distinct elements of $U$ correct if
\begin{itemize}
    \item $w_i \nsucc w_{i+1}$ for $i = 1, 2, \dots, k - 1$ and
    \item for each $j = 2, \dots, k$, there exists $i < j$ such that $w_i \sim w_j$.
\end{itemize}
We denote by $C_k^U$ the set of length-k correct sequences in $U$. For a partition $\lambda=(\lambda_1,\ldots, \lambda_r)$ we denote by $C_\lambda^U$ the set of length $\lambda_1+\ldots +\lambda_r$-sequences in $U$, which are concatenations of a length-$\lambda_1$-correct sequence with a length-$\lambda_2$-correct sequence with a ... with a length-$\lambda_r$-correct sequence.
\end{definition}

\begin{theorem}[Szenes-Rok, \cite{szenesrok}]\thlabel{paunov}
For any $k\le N$
$$p_k^{U} = \sum\limits_{w \in C_{k}^U}w_1 \cdot ... \cdot w_{k} \in \mathbb{N}[U]$$
\end{theorem}
Furthermore for any partition $\lambda=(\lambda_1,\ldots, \lambda_r)$ we have 
\begin{align}\label{powercorrect}
p_{\lambda}^{U} = \prod\limits_{i=1}^{r}p_{\lambda_i}^{U} = \prod\limits_{i=1}^{r}\sum\limits_{w \in C_{\lambda_i}^U}w_1 \cdot \ldots \cdot w_{\lambda_i} = \sum_{w \in C_{\lambda}^U}w_1 \cdot w_1 \cdot \ldots \cdot w_{N}
\end{align}
This leads to 
\begin{proposition}\thlabel{1coeff}
$c_N^U = \# C_N^U \geq 0$
\end{proposition}
\begin{proof}
Using the identity $p_N = \sum\limits_{i = 1}^{\infty} x_i^N = m_N$ we apply the $G$-homomorphism to get $m_N^{U} = p_N^{U}$ and thus $c_N = [m_N^U, v_1 \cdot v_2 \cdot ... \cdot v_N] = \# C_N^U \geq 0$ by \ref{paunov}.
\end{proof}

\begin{theorem}[Szenes-Rok, \cite{szenesrok}]\label{thm:szenesrok} Let $U$ be a UIO on $N$ intervals. Then $\# E_k^U=\# C_k^U$ for any $k \le N$ and hence by \thref{1coeff}
\[
c^U_N =\# E_N^U
\]
\end{theorem}

\subsection{Escher-pairs} 
Let $\lambda=(n,k)$ be a fixed pair with $n+k\le N$. Apply the $G$-homomorphism to the identity
$$m_{(n,k)} = p_np_k - p_{n+k}$$
in $\Lambda$ and we get 
$$m_{(n,k)}^{U} = p_n^{U}p_k^{U} - p_{n+k}^{U}$$
Using Theorem \ref{thm:szenesrok} we get
\begin{equation}\label{2formula}
c_{(n,k)}^U = \#E_n^U \cdot \#E_k^U - \# E_{n+k}^U
\end{equation}
In fact, it is not hard to see that we can work with disjoint Escher pairs, hence we arrive at
\begin{theorem}[Stanley coeff formula for pairs]\label{thm:pairformula}
 \[c_{(n,k)}^U = \#E_{n,k}^U - \# E_{n+k}^U\]  
\end{theorem}
To prove non-negativity of $c_{(n,k)}^U$, Szenes and Rok  \cite{szenesrok} construct an injective function 
$$\phi:E_{n+k}^U  \to E_{(n,k)}^U$$
We go through the key steps of the proof to show the intuition behind our core representations $\tau_k$. 
The idea is to find two functions 
\[\phi:E_{n+k}^U  \to E_{(n,k)}^U,\ \  \psi:E_{n,k}^U  \to E_{n+k}^U \text{ such that } \psi \circ \phi=id\] 
Let $w \in E_{n+k}^U$. By Proposition \ref{thereissubescher}, $w$ has a first splitting point $L = \mathbf{SEStart}(w,k)$, so that $v = [w_{L+1}, ..., w_{L+k}]$ and $u = [w_{L+k+1}, ..., w_{L+k+n}]$ are $k$ and $n$ Eschers respectively. It's natural to define $\phi$ as $w \mapsto (u,v)$, but the problem is that it is not always the case that $(u,v)$ has an insertion point which we could use to concatenate $u$ and $v$ to define $\psi$. But it turns out that there are cyclic permutations $u' \in [u]$, $v' \in [v]$ s.t. $(u', v')$ has an insertion point. Namely, pick $u' \in [u]$ s.t. $u'_L = u_{n-1}$ and $v' \in [v]$ s.t. $v'_{L} = v_{k-1}$, i.e. 
\begin{equation}\label{rotate}
u'_i = u_{i+n-L-1} \text{ and }v'_i = v_{i+k-L-1}
\end{equation}
Define 
\[\phi:E_{n+k}^U \to E_{n,k}^U, \ \ \ \phi(w) = (u',v')\]. 
Then we prove that 
\begin{lemma}\thlabel{back}
For any $w \in E_{n+k}^U$ the pair $(u', v') = \phi(w)$ has an insertion point at $L=\mathbf{SEStart}(w,k)$. Furthermore $u'+_L v' \in [w]$
\end{lemma}
\begin{proof}

$u'$ and $v'$ can be concatenated at $L$ since
\begin{align*}
u'_L = u_{n-1} = w_{L+k+n} = w_L \rightarrow w_{L+1} = v_0 = v'_{L+1}, \\ 
v'_L = v_{k-1} = w_{L+k} \rightarrow w_{L+k+1} = u_0 = u'_{L+1}
\end{align*}
and $u'+_L v'$ is the same cycle-type as $w$ since
\begin{align*}
u'+_L v' = [u'_q, ..., u'_{L}, v'_{L+1}, ..., v'_{L+k}, u'_{L+1}, ..., u'_{q-1}] \\
= [u_{q-L-1}, ..., u_{n-1}, v_0, ..., v_{k-1}, u_0, ..., u_{q-L-2}] \\
\sim [v_0, ..., v_{k-1}, u_0, ..., u_{n-1}] = [w_{L+1}, ..., w_{L+k}, w_{L+k+1}, ..., w_{L+k+n}] \in [w]
\end{align*}
\end{proof}
Now let's try to construct $\psi$, we already know $u'+_L v' \sim w$, so for them to be equal, we only need their starting point to be the same. The first element of $u'+_L v'$ will by definition always be from $u'$, i.e. from $u$, but depending on where the $k$-subescher $v$ starts in $w$, it could be that $w_0 \in v$. So if we want to recover $w$ from $u'$ and $v'$ we need to change the starting point of $u'+_L v'$ in one of the 2 cases. 

To find the criterion to decide in what case we are, note that $u_i = w_{L+k+1+i}$ for $0 \leq i < n$ and $v_i = w_{L+1+i}$ for $0 \leq i < k$. So $w_0 \in u$ can only happen if $u_{n-L-1} = w_0$, so the criterion is $n-L-1 \in [0, n-1] \iff 0 \leq L \leq n-1$. Similarly we get $w_0 \in v \iff v_{n+k-L-1}=w_0 \iff n \leq L < n+k$.

If $L < n$ we get by (\ref{rotate}) that 
$$(u' +_L v', L)_0 = u'_0 = u_{L-n-1} = w_0 \Rightarrow u' +_L v' = w$$

If $L \geq n$ we get by (\ref{rotate}) that 
$$(u' +_L v')^{v'_n}_0 =  v'_n = v_{n+k-L-1} = w_0 \Rightarrow (u' +_L v')^{v'_n} = w$$ 

Our discussion can be summarized in
\begin{lemma}\thlabel{backer}
Let $w \in E_{n+k}$,$(u,v) = \phi(w)$ and $L = FS(w)$ then $$w = \begin{cases} 
u +_L v & \text{if } L < n \\
(u +_L v)^{v_n} & \text{if } L \geq n
\end{cases}$$
\end{lemma}

This inspires our definition of $\psi:$
\begin{definition}
Fix some dummy $w_0 \in E_{n+k}$, define $\psi: E_{n,k} \to E_{n+k}$ as 
$$\psi((u,v)) = 
\begin{cases} 
w_0 & \text{if } (u,v) \text{ has no insertion point} \\
u +_{FI(u,v)} v & \text{if } FI(u,v) < n \\
(u +_{FI(u,v)} v)^{v_n} & \text{if } FI(u,v) \geq n
\end{cases}$$
\end{definition}

\begin{theorem}\thlabel{id}
$\psi \circ \phi$ is the identity
\end{theorem}
We have the following characterization of $E_{n,k} \setminus im(\phi)$:
\begin{corollary}
The complement of the image of $\phi$ can be seen as
\begin{align*}
im(\phi)^C = \set{(u,v) \in E_{n,k}}{(u,v) \text{ has no insertion point or } \phi \circ \psi ((u,v)) \neq (u,v) }
\end{align*}
\end{corollary}
\begin{proof}
Let $(u,v) \in im(\phi)$, pick $w \in E_{n+k}$ s.t. $\phi(w) = (u,v)$ then $\phi \circ \psi (u,v) = \phi(w) = (u,v)$. On the other hand $\phi \circ \psi ((u,v)) = (u,v) \in im(\phi)$.
\end{proof}

\subsection{Escher-tuples} 
Let $\lambda=(\lambda_1,\ldots, \lambda_r)$ a partition. 
Using Stanley's G-homomorphism, we can come up with a similar formula for $c_\lambda^U$ using Escher-tuples as in to that in Theorem \ref{thm:pairformula}. Indeed, the symmetric polynomial $m_{(\lambda_1,\ldots, \lambda_r)} \in \Lambda$ can be uniquely written in the total sum basis as   
$$m_{(\lambda_1,\ldots, \lambda_r)} = \sum_{\tau} d_{\tau}^\lambda p_\tau$$
with integer coefficients $d_\tau^\lambda$. Applying the $G$-homomorphism we obtain
$$m_{(\lambda_1,\ldots, \lambda_r)}^U = \sum_{\tau} d_{\tau}^\lambda p_\tau^U$$
and by Theorem \ref{thm:szenesrok} again, we get
\begin{theorem}[Stanley coeff formula]\label{thm:coefformula}
 \[c_{\lambda}^U = \sum_{\tau} d_{\tau}^\lambda \cdot \# E_\tau^U.\]  
\end{theorem}
We state this formula for triple and quadruple partitions as corollary, for reference.
\begin{corollary}[Stanley triple coeff formula]\label{thm:tripleformula}
 \[c_{(n,k,l)}^U = \#E_{(n,k,l)}^U+ 2\cdot \#E_{(n+k+l)}^U-\#E_{(n+k,l)}^U - \#E_{(n+l,k)}^U - \#E_{(k+l,n)}^U \]  
\end{corollary}
Using the pair formula in Theorem \ref{thm:pairformula} we can rewrite this as 
\[c_{(n,k,l)}^U = \#E_{(n,k,l)}^U + \#E_{(n+k+l)}^U-c_{(n+k,l)}^U - c_{(n+l,k)}^U - c_{(k+l,n)}^U  \]
\begin{corollary}[Stanley quadruple coeff formula]\label{thm:quadrupleformula}
 \begin{align*}
 c_{(n,k,l,m)}^U = & \#E_{(n,k,l,m)}^U-\\
 &-\#E_{(n+k,l,m)}^U - \#E_{(n+l,k,m)}^U - \#E_{(n+m,k,l)}^U-\#E_{(k+l,n,m)}^U-\#E_{(k+m,n,l)}^U-\#E_{(l+m,n,k)}^U+\\
 &+ \#E_{(n+k,l+m)}^U+\#E_{(n+l,k+m)}^U +\#E_{(n+m,k+l)}^U+\\
 &+ 2*(\#E_{(n+k+l,m)}^U +\#E_{(n+k+m,l)}^U+\#E_{(n+l+m,k)}^U+\#E_{(k+l+m,n)}^U-\\
 & -6*\#E_{(n,k,l,m)}^U
\end{align*}
\end{corollary}

\newpage
\section{Machine Learning}

The message of Theorem \ref{thm:coefformula} is that the Stanley coefficient $c_\lambda$ should count $\lambda$-Escher tuples satisfying some specific splitting and concatenation properties. Following the conventions and notations of \S \ref{sec:mainresults}, for any partition $\lambda = (\lambda_1,\ldots, \lambda_r)$ we define the core representation 
\[\tau_\lambda: E_\lambda^U \to \mathbb{Q}^{\xi(\lambda)}\]
which roughly sends a $\lambda$-Escher-tuple $\mathbf{u}=(u_1,\ldots u_r)$ to its core vector of dimension $\xi(\lambda)$, collecting $\mathbf{SEStart}(v,w), \mathbf{SEEnd}(v,w)$ and $\mathbf{FirstIns}(v,w)$ points for all possible $v,w$ pairs which come from concatenation of Eschers in the tuple. Coordinates of the core vector are called cores, and the number $\xi(\lambda)=\xi(|\lambda|)$ will depend only on the length of $\lambda$. 

As we have seen in \S \ref{sec:mainresults}, the delicate issue is how to define those coordinates of the representation where there is no concatenation point. For example, how do we define $\mathbf{SEStart}(u_iu_j, u_k)$, $\mathbf{SEEnd}(u_iu_j, u_k)$, and $\mathbf{FirstIns}(u_iu_j, u_k)$ when $\mathbf{I}(u_i, u_j) = \emptyset$, and therefore we can not concatenate $u_i$ with $u_j$? 
\[\tau^r: E_{(\lambda_1,\ldots, \lambda_r)}^U \to \mathbb{Q}^{\xi(\lambda)}\]
\[
    \tau^r(u_1,\ldots, u_r)=(0,\tau^2(u_i,u_j)[1:],
    \tilde{\tau}^2(u_iu_j,u_k))
\]
where $1\le i<j<k \le r$ and we do not add further core coordinates, corresponding to longer concatenations, because (according to our experiments) they do not result in better condition graphs. For the definition of $\tau^2(u_i,u_j)[1:]$ and $\tilde{\tau}^2(u_iu_j,u_k)$ see \eqref{def:tildetau}.

We call a permutation of the $\xi(\lambda)$ entries of the core vector a sorting permutation, if it moves the entries to descending order. The sorting permutation is not unique in case there are repeated entries. For a permutation $\tau \in S_{\xi(\lambda)}$ we denote by $E_\lambda^U(\tau)$ the set of $\lambda$-tuples whose sorting permutation is $\tau$. 

Following extensive experiments with Gurobi Software \cite{gurobi} involving Escher tuples, triples and quadruples, we arrived to the following surprising conjecture, which says that the Stanley coefficients are determined by the sorting orders, not the actual value of the core coordinates.

\begin{conjecture}\label{conj:decisiontree}
For any partition $\lambda$ there is a set of sorting orders 
\[T_\lambda=\{\tau_1,\ldots, \tau_{N(\lambda)})\},\]
independent of $U$, such that 
\[c_\lambda^U = \# (E_\lambda^U(\tau_1) \cup 
\ldots \cup E_\lambda^U(\tau_N)).\]
That is, the Stanley coefficient $c_\lambda$ counts the number of $\lambda$-Escher tuples with sorting core order in $T_\lambda$. We say that $T_\lambda$ characterises the Stanley coefficient $c_\lambda$
\end{conjecture}

While attempting to describe the Stanley coefficients through the use of characterizing sets, several significant challenges arose, which we can summarise as follows.

\begin{enumerate}
    \item \textbf{Exponential Growth of Characterizing Sets}: The size \( N(\lambda) \) of the characterizing set increases drastically, even for small cases like triples. This rapid growth makes the analysis computationally difficult and inefficient.
    \item \textbf{Dependence on Partition Structure}: The characterizing set is not solely determined by the length \( |\lambda| \) of the partition but depends intricately on the specific structure of the partition \( \lambda \). This complicates generalization and creates additional complexity when dealing with different partitions of the same length.
    \item \textbf{Subset-Sum Problem and Decision Tree Complexity}: Identifying the appropriate characterizing sets can be viewed as a specialized instance of the subset-sum problem in combinatorics, which is known to be NP-hard. Attempts to represent these sets using decision trees have resulted in extremely complicated and large structures. These trees are not only cumbersome to compute but also impractical to use for further theoretical analysis.
\end{enumerate}
Due to these challenges, we were unable to find a workable description of characterizing sets even for partitions of length 3, which highlights the difficulty in scaling this approach for more general cases.

Our strategy, instead, was to describe $T_\lambda$ as a set of partitions $\tau \in S_{l(\lambda)}$ satisfying a Boolean expressions of the form 
\[\mathrm{Condition}_1 \textbf{ OR } \mathrm{Condition}_2 \textbf{ OR } \ldots \textbf{ OR } \mathrm{Condition}_r\]
where each condition has the form 
\[\text{Condition  = } \tau(i_1)<\tau(j_1) \textbf{ AND } \ldots \textbf{ AND } \tau(i_s)<\tau(j_s)\]
We encoded $\mathrm{Condition}_i$ as a graph $G^i_\lambda$ on the vertices $\{1,\ldots, \xi(\lambda)\}$, with a directed edge from $i_t$ to $j_t$ for $t=1,\ldots, r$, and call 
\[G_\lambda= G^1_\lambda \cup \ldots \cup G^r_\lambda\]
a condition graph. Due to the coding practice, we often refer to the number of Conditions as the number of rows of our condition graph.  


\subsection{The RL algorithm}\label{adamsec}

Our policy-based reinforcement learning model is an adapted version of the cross-entropy method for graphs, as developed by Adam Zsolt Wagner \cite{wagner}. The deep NN learns a probability distribution that guides the agent’s decisions during graph construction.

The RL agent learns to identify optimal condition graphs. Specifically, after fixing the number of rows \( r \), we initialized \( r \) copies \( G_1, \ldots, G_r \) of the empty graph on \( \xi(\lambda) \) vertices. The edge set of each graph \( G_i \) is ordered as \( e_1^i, \ldots, e_m^i \), where \( m = \binom{\xi(\lambda)}{2} \), and we defined a sequential edge ordering across the \( r \) graphs as follows:
\[
e_1^1, e_1^2, \ldots, e_1^r, e_2^1, e_2^2, \ldots, e_2^r, \ldots, e_m^r.
\]
At each step, the RL agent processes this ordered list of edges and adds the corresponding edge to the current graph \( G_\lambda \) based on a policy learned by a deep neural network (NN). Each run consists of \( r \cdot \binom{\xi(\lambda)}{2} \) steps, where the agent decides whether or not to add each edge.

The neural network takes two binary vectors as input:
\begin{enumerate}
    \item Game Turn Vector: A binary vector that tracks the number of moves made in the current game, with all entries set to 0 except the one corresponding to the current turn set to 1.
    \item Graph State Vector: A binary vector representing the current state of the graph. The \( i \)-th entry is set to 1 if the edge corresponding to that entry was added during the \( i \)-th turn, and 0 otherwise.
\end{enumerate}

The output of the NN is a probability distribution over \(\{0, 1\}\). The agent samples from this distribution to make its decision for the current step: if the sampled value is 1, the corresponding edge is added to the graph.

The RL agent learns the game by running batches of agents through a series of complete games. After each game, the agents are evaluated based on their performance, and the top-performing individuals--typically the top 10\%--are selected to update the neural network. It is important to note that there is a single shared NN used by all agents during training. The input-output pairs (game state at step \( i \), decision at step \( i \)) from the top agents are used to adjust the NN parameters via the cross-entropy method. This selection mechanism encourages the model to replicate the successful behaviors of the best-performing agents.

Additionally, a small percentage of the all-time best agents are retained throughout training and used in future updates. This approach ensures that the NN continues to learn from historically strong strategies, thereby improving the likelihood of finding optimal solutions.

$\\ \\ $
We modified Wagner’s cross-entropy method at the following crucial points: 

\begin{enumerate}
    \item Multi-Type Edges: While Wagner’s method uses binary edge choices, our graphs allow edges to take multiple types. Accordingly, we modify the probability distribution to output values in \(\{0, \ldots, K-1\}\), where \( K \) is the number of possible edge types. Furthermore we use categorical cross-entropy instead of binary cross-entropy for training and our Graph State Vector uses one-hot encoding for every K consecutive bits to encode a value in \(\{0, \ldots, K-1\}\).  Concretely we use K=3 to express that either the source vertex is smaller than the target vertex, or greater equal to the target vertex or that the relation is irrelevant.

    \item Edge Selection: Not every two vertices have a meaningful relationship, so out of the full \( r \cdot m \) edges we will only pick some and reduce the size of our ordered list of edges. Our selection was heavily guided by mathematical intuition and was necessary since the training otherwise would be stuck in local minima when using all edges. Reducing the edges also improves the running time of the algorithm linearly.
    
    \item State Memorization: To speed up computation, we cache the scores of previously encountered graph states. This allows us to reuse scores instead of recomputing them when the same graph configuration is encountered again, significantly reducing computational overhead in later stages of training.
    
\end{enumerate}

\section{Implementation}

The complete python code implementation can be found at
\begin{center}
\url{https://github.com/berczig/PositivityConjectures}.
\end{center}

To generate the necessary training data for our model, we follow a multi-step process involving combinatorial generation and brute-force classification. The key steps are outlined below:\\

\textbf{Step 1: Generating UIO graphs}
We first generate all UIOs of length \( N \), where \( N \leq 10 \). For each UIO, we construct the associated incomparability graph. This graph is represented as an \( N \times N \) matrix, where for each pair \( (i, j) \), the matrix entry specifies the relation between \( u_i \) and \( u_j \) for the UIO \( (u_0, \ldots, u_{N-1}) \). The incomparability graph is used as a structural feature for downstream tasks.\\

\textbf{Step 2: Generating $\lambda$-Eschers}
Using a brute-force search approach, we generate all possible tuples of UIOs and identify which tuples correspond to Escher structures. The process systematically checks each tuple to determine whether or not it satisfies the Escher condition. This classification serves as a binary label for the dataset.\\

\textbf{Step 3: Calculating Coefficients \( c_\lambda^U \)}
For each valid UIO \( U \) and partition \( \lambda \), we calculate the corresponding coefficient \( c_\lambda^U \) using Theorem \ref{thm:pairformula} and Corollary \ref{thm:tripleformula}, \ref{thm:quadrupleformula}. The generated dataset includes coefficients for all tuples \( (\lambda, U) \), where \( \lambda = (\lambda_1, \ldots, \lambda_r) \) and \( U = (u_1, \ldots, u_N) \), subject to the constraints:
\[
2 \leq r \leq 4, \quad \lambda_1 + \cdots + \lambda_r \leq N \leq 10.
\]
With a C++ implementation, we successfully generated a dataset of approximately 100,000 tuples, enabling efficient training of our model.\\

\textbf{Step 4: Calculating Cores and Core-Types}
For each UIO \( U \), we compute all its core vectors \( \tau_\lambda(E_{\lambda}^U) \), classifying all core vectors into emerging core-types. The dataset stores the distinct core-types along with their respective counts. Specifically, we iterate over all cores, and when a new core-type is encountered, we store it. If a core-type has been observed previously, we simply increment the counter associated with that core-type.

The number of core-types is significantly smaller than the total number of \( \lambda \)-Escher tuples, which enables optimization. When evaluating a condition graph, we iterate over the distinct core-types, summing the contributions of all matching core-types with their respective multiplicities:
\begin{equation}\label{cgu}
c_G^U = \sum_{\substack{ w \in E_{\lambda}^U\\ \tau^r(w) \text{ satisfies } G}} 1 = \sum_{\substack{ w \in \tau^r(E_{\lambda}^U)\\ w \text{ satisfies } G}} \text{\#cores with core-type } w.
\end{equation}
By reducing the evaluation to core-types, we can streamline computation, especially for large-scale experiments.

In our model we calculate all coefficients simultaneously, this is faster since two UIOs can share the same core-type and by batching the computation we only have to check if the core-type satisfies a condition graph once. We collect all distinct core-types 

\[B = \bigcup\limits_{\text{UIO }U, |U|=n} \tau^r(E_{\lambda}^U)\] 

and calculate the coefficient prediction vector (recall $C_n$ stands for the $n$th Catalan number, which is the number of length n UIOs)

\begin{equation}\label{cgu_batched}
c_G = \sum_{\substack{ w \in B\\ w \text{ satisfies } G}} (\text{\#cores with core-type } w \text{ in }U_1,  \cdots, \text{\#cores with core-type } w \text{ in }U_{C_n}).
\end{equation}

\subsection{Score Function}  

Given a partition \( \lambda \) we apply the same score function for all UIOs of the same size, i.e. having the same number of intervals. Let \( 0 < n \) denote the number of intervals, and \( C_n \), the n'th Catalan number, which is the number of UIOs of size n. We let $x \in \mathbb{N}^{C_n}$
denote the predicted coefficients vector and let $y \in \mathbb{N}^{C_n}$ denote the true Stanley coefficients vector. Then we define the score function as
\begin{equation}\label{scorefunction}
score_1(x) = - ||x-y||_1 - \mathbf{1}(\text{has trivial row}) \cdot 10000 -  
num\_edges \cdot EdgePenalty
\end{equation}    

In the modified version where we search for a lower bound of the Stanley coeffiecient, we set the score to be a large negative number in case there is a UIO with where predicted coeff \( > \) true coeff:

\begin{align}\label{scorefunction_only_negative}
score_2(x) = 
\begin{cases}
    - 5000, & \text{if } x_i > y_i \text{, for any }i \geq 0 \\
    score_1(x), & else
\end{cases}
\end{align}

Notice, we have opted for a negative score function because our objective is to maximize the score, with the optimal value set at \( 0 \).

\subsection{Hyperparameters}
Wagner's implementation includes several hyperparameters that govern the learning process, such as the learning rate, the number of graphs per batch, and the selection percentage of the top-performing individuals, among others. We retained most of these hyperparameters in our experiments. The learning rate was primarily set to either \( 0.01 \) or \( 0.05 \). After conducting numerous experiments, we found that a batch size of \( 600 \) graphs yielded comparable results to a larger batch size of \( 1500 \) graphs, making it a more efficient choice. Moreover, we employed one-hot encoding for graph representations.

A significant speedup in our approach was achieved by focusing on a random subset of the UIOs. If the algorithm fails to find a condition graph for this subset, it is unlikely to succeed when evaluated against all UIOs. This strategic reduction not only enhances computational efficiency but also improves the likelihood of successful evaluations, and allowed us to work with longer UIOs.

\subsection{Results}

This section presents the results obtained using the RL agent, summarized in tables. Each table provides an overview of the mismatches between the Stanley coefficients and the counts associated with a given condition graph. The mismatches are evaluated in two different ways, in individual tables

\begin{enumerate}
    \item \textbf{Tables with number of correctly predicted UIOs}: The entry $A/B$ indicates that, out of a total of $B$ UIOs, the condition graph correctly predicts the coefficient for $A$ UIOs.
    \item \textbf{Tables with total absolute error}: The entry $A/B$ means that the total sum of the Stanley coefficients for all UIOs is $B$, while
    \[A=\sum_{U} |\text{predicted coeff for U} - \text{ real coeff for U}|\]
\end{enumerate}

Entries highlighted in blue indicate that the predicted coefficient is less than or equal to the actual coefficient for all UIOs of the given length, meaning the condition graph provides a lower bound for the coefficients.
In contrast, entries highlighted in orange indicate the opposite: the predicted coefficient is greater or equal to the real for all UIOs. Magenta entries indicate mixed results, and non-higlighted entries indicate perfect match. 

\newpage

\begin{table}[h!]
\centering
\resizebox{0.7\textwidth}{!}{
\begin{tabular}{|l||*{6}{c|}}
\hline
\multicolumn{7}{|c|}{Total absolute error/sum of real coefficients for all UIO} \\
\hline
\backslashbox{Partition}{UIO length}
&\makebox[3em]{4}&\makebox[3em]{5}&\makebox[3em]{6}
&\makebox[3em]{7}&\makebox[3em]{8}&\makebox[3em]{9}\\\hline\hline
(2, 2) & 0/12 & 0/192 & 0/1876 & 0/14496 & 0/97436 & 0/597056 \\
\hline
(3, 1) & 0/30 & 0/440 & 0/4044 & 0/29852 & 0/193626 & 0/1152912 \\
\hline
(3, 2) &  & 0/75 & 0/1446 & 0/16517 & 0/145828 & 0/1100405 \\
\hline
(4, 1) &  & 0/259 & 0/4590 & 0/49193 & 0/413056 & 0/2992325 \\
\hline
(3, 3) &  &  & 0/414 & 0/9630 & 0/128796 & 0/1302006 \\
\hline
(4, 2) &  &  & 0/746 & 0/16532 & 0/213152 & 0/2093462 \\
\hline
(5, 1) &  &  & 0/2820 & 0/58706 & 0/719904 & 0/6784258 \\
\hline
(4, 3) &  &  &  & 0/3903 & 0/103326 & 0/1550199 \\
\hline
(5, 2) &  &  &  & 0/9595 & 0/240366 & 0/3450237 \\
\hline
(6, 1) &  &  &  & 0/36639 & 0/877158 & 0/12124473 \\
\hline
(4, 4) &  &  &  &  & 0/32008 & 0/971824 \\
\hline
(5, 3) &  &  &  &  & \color{cyan} 1/52560\color{black} & \color{cyan} 28/1541010\color{black} \\
\hline
(6, 2) &  &  &  &  & 0/146100 & 0/4093944 \\
\hline
(7, 1) &  &  &  &  & 0/550914 & 0/14917146 \\
\hline
(5, 4) &  &  &  &  &  & \color{cyan} 6/397195\color{black} \\
\hline
(6, 3) &  &  &  &  &  & 0/877977 \\
\hline
(7, 2) &  &  &  &  &  & 0/2534175 \\
\hline
(8, 1) &  &  &  &  &  & 0/9395415 \\
\hline
\end{tabular}}

\resizebox{0.7\textwidth}{!}{
\begin{tabular}{|l||*{6}{c|}}
\hline
\multicolumn{7}{|c|}{Number of correctly predicted/number of all UIOs} \\
\hline
\backslashbox{Partition}{UIO length}
&\makebox[3em]{4}&\makebox[3em]{5}&\makebox[3em]{6}
&\makebox[3em]{7}&\makebox[3em]{8}&\makebox[3em]{9}\\\hline\hline
(2, 2) & 14/14 & 42/42 & 132/132 & 429/429 & 1430/1430 & 4862/4862 \\
\hline
(3, 1) & 14/14 & 42/42 & 132/132 & 429/429 & 1430/1430 & 4862/4862 \\
\hline
(3, 2) &  & 42/42 & 132/132 & 429/429 & 1430/1430 & 4862/4862 \\
\hline
(4, 1) &  & 42/42 & 132/132 & 429/429 & 1430/1430 & 4862/4862 \\
\hline
(3, 3) &  &  & 132/132 & 429/429 & 1430/1430 & 4862/4862 \\
\hline
(4, 2) &  &  & 132/132 & 429/429 & 1430/1430 & 4862/4862 \\
\hline
(5, 1) &  &  & 132/132 & 429/429 & 1430/1430 & 4862/4862 \\
\hline
(4, 3) &  &  &  & 429/429 & 1430/1430 & 4862/4862 \\
\hline
(5, 2) &  &  &  & 429/429 & 1430/1430 & 4862/4862 \\
\hline
(6, 1) &  &  &  & 429/429 & 1430/1430 & 4862/4862 \\
\hline
(4, 4) &  &  &  &  & 1430/1430 & 4862/4862 \\
\hline
(5, 3) &  &  &  &  & \color{cyan} 1429/1430\color{black} & \color{cyan} 4842/4862\color{black} \\
\hline
(6, 2) &  &  &  &  & 1430/1430 & 4862/4862 \\
\hline
(7, 1) &  &  &  &  & 1430/1430 & 4862/4862 \\
\hline
(5, 4) &  &  &  &  &  & \color{cyan} 4858/4862\color{black} \\
\hline
(6, 3) &  &  &  &  &  & 4862/4862 \\
\hline
(7, 2) &  &  &  &  &  & 4862/4862 \\
\hline
(8, 1) &  &  &  &  &  & 4862/4862 \\
\hline
\end{tabular}}

\includegraphics[width=4.5cm]{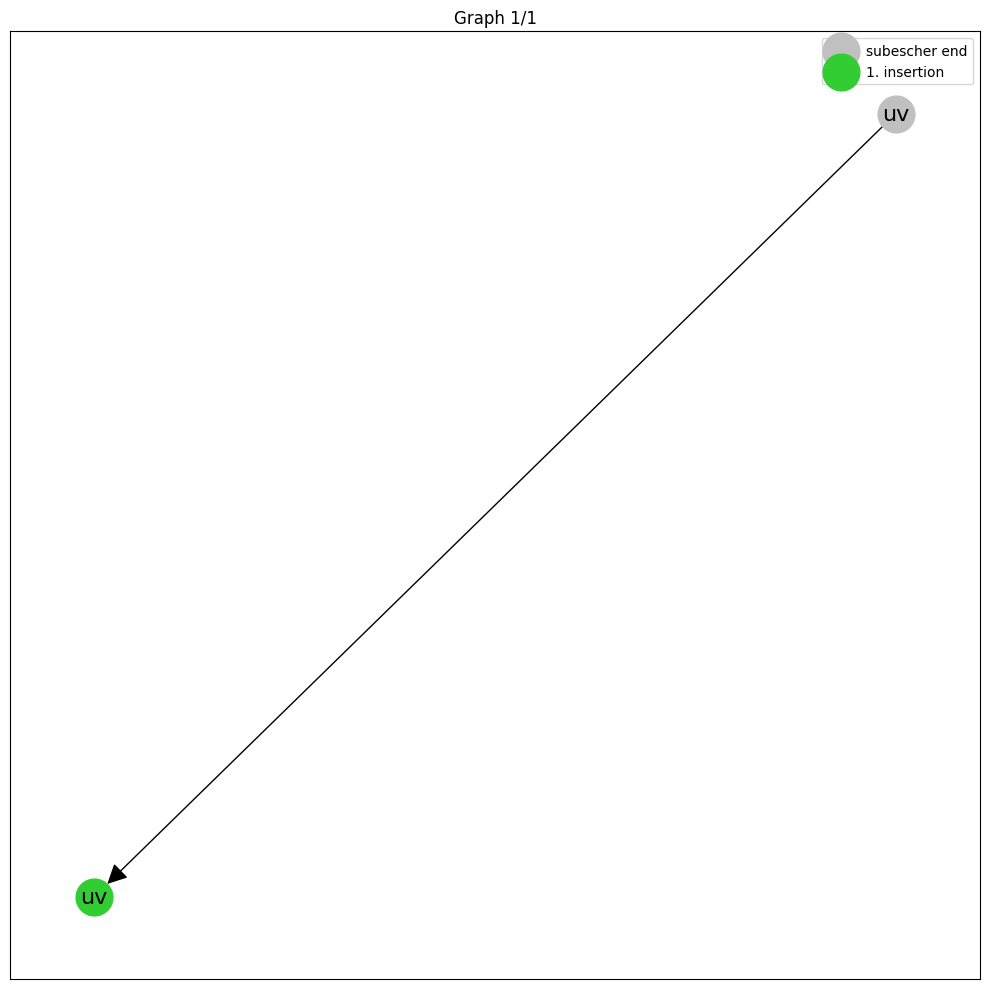}
\centering

\caption{Mismatch statistics for length 2 partitions for the model [SEEnd $<$ FirstIns], corresponding to the graph at bottom. Green entries indicate that this model gives lower bound for the Stanley coefficients even when mismatches occur.}

\label{tab:pairs}
\end{table}

\clearpage

\begin{table}[h!]
\centering
\resizebox{0.65\textwidth}{!}{
\begin{tabular}{|l||*{6}{c|}}
\hline
\multicolumn{7}{|c|}{Total absolute error/sum of real coefficients for all UIO} \\
\hline
\backslashbox{Partition}{UIO length}
&\makebox[3em]{4}&\makebox[3em]{5}&\makebox[3em]{6}
&\makebox[3em]{7}&\makebox[3em]{8}&\makebox[3em]{9}\\\hline\hline
(2, 1, 1) & 0/16 & 0/256 & 0/2552 & 0/20304 & 0/141088 & 0/895072 \\
\hline
(2, 2, 1) &  & 0/42 & 0/860 & 0/10454 & 0/98088 & 0/784470 \\
\hline
(3, 1, 1) &  & 0/106 & 0/2052 & 0/23846 & 0/215620 & 0/1671830 \\
\hline
(2, 2, 2) &  &  & 0/108 & 0/2712 & 0/39216 & 0/427572 \\
\hline
(3, 2, 1) &  &  & \color{orange} 5/286\color{black} & \color{orange} 100/6872\color{black} & \color{orange} 1184/95536\color{black} & \color{orange} 10812/1006152\color{black} \\
\hline
(4, 1, 1) &  &  & 0/1024 & 0/23148 & 0/306112 & 0/3091580 \\
\hline
(3, 2, 2) &  &  &  & 0/742 & 0/21372 & 0/347850 \\
\hline
(3, 3, 1) &  &  &  & \color{orange} 16/1824\color{black} & \color{orange} 400/50940\color{black} & \color{orange} 5736/805620\color{black} \\
\hline
(4, 2, 1) &  &  &  & \color{magenta} 68/3034\color{black} & \color{magenta} 1503/82612\color{black} & \color{magenta} 19623/1279950\color{black} \\
\hline
(5, 1, 1) &  &  &  & 0/12472 & 0/322284 & 0/4778064 \\
\hline
(3, 3, 2) &  &  &  &  & \color{orange} 54/4776\color{black} & \color{orange} 1560/156156\color{black} \\
\hline
(4, 2, 2) &  &  &  &  & 0/8344 & 0/267792 \\
\hline
(4, 3, 1) &  &  &  &  & \color{orange} 724/18758\color{black} & \color{orange} 20117/587564\color{black} \\
\hline
(5, 2, 1) &  &  &  &  & \color{magenta} 1083/41706\color{black} & \color{magenta} 25703/1265520\color{black} \\
\hline
(6, 1, 1) &  &  &  &  & 0/180144 & 0/5237856 \\
\hline
(3, 3, 3) &  &  &  &  &  & 0/28620 \\
\hline
(4, 3, 2) &  &  &  &  &  & \color{orange} 1173/51836\color{black} \\
\hline
(4, 4, 1) &  &  &  &  &  & \color{orange} 3236/177434\color{black} \\
\hline
(5, 2, 2) &  &  &  &  &  & 0/124470 \\
\hline
(5, 3, 1) &  &  &  &  &  & \color{magenta} 12102/265838\color{black} \\
\hline
(6, 2, 1) &  &  &  &  &  & \color{magenta} 17809/681792\color{black} \\
\hline
(7, 1, 1) &  &  &  &  &  & 0/2987556 \\
\hline
\end{tabular}}

\resizebox{0.65\textwidth}{!}{
\begin{tabular}{|l||*{6}{c|}}
\hline
\multicolumn{7}{|c|}{Number of correctly predicted/number of all UIOs} \\
\hline
\backslashbox{Partition}{UIO length}
&\makebox[3em]{4}&\makebox[3em]{5}&\makebox[3em]{6}
&\makebox[3em]{7}&\makebox[3em]{8}&\makebox[3em]{9}\\\hline\hline
(2, 1, 1) & 14/14 & 42/42 & 132/132 & 429/429 & 1430/1430 & 4862/4862 \\
\hline
(2, 2, 1) &  & 42/42 & 132/132 & 429/429 & 1430/1430 & 4862/4862 \\
\hline
(3, 1, 1) &  & 42/42 & 132/132 & 429/429 & 1430/1430 & 4862/4862 \\
\hline
(2, 2, 2) &  &  & 132/132 & 429/429 & 1430/1430 & 4862/4862 \\
\hline
(3, 2, 1) &  &  & \color{orange} 129/132\color{black} & \color{orange} 391/429\color{black} & \color{orange} 1157/1430\color{black} & \color{orange} 3359/4862\color{black} \\
\hline
(4, 1, 1) &  &  & 132/132 & 429/429 & 1430/1430 & 4862/4862 \\
\hline
(3, 2, 2) &  &  &  & 429/429 & 1430/1430 & 4862/4862 \\
\hline
(3, 3, 1) &  &  &  & \color{orange} 425/429\color{black} & \color{orange} 1369/1430\color{black} & \color{orange} 4344/4862\color{black} \\
\hline
(4, 2, 1) &  &  &  & \color{magenta} 404/429\color{black} & \color{magenta} 1202/1430\color{black} & \color{magenta} 3493/4862\color{black} \\
\hline
(5, 1, 1) &  &  &  & 429/429 & 1430/1430 & 4862/4862 \\
\hline
(3, 3, 2) &  &  &  &  & \color{orange} 1416/1430\color{black} & \color{orange} 4651/4862\color{black} \\
\hline
(4, 2, 2) &  &  &  &  & 1430/1430 & 4862/4862 \\
\hline
(4, 3, 1) &  &  &  &  & \color{orange} 1309/1430\color{black} & \color{orange} 3824/4862\color{black} \\
\hline
(5, 2, 1) &  &  &  &  & \color{magenta} 1268/1430\color{black} & \color{magenta} 3662/4862\color{black} \\
\hline
(6, 1, 1) &  &  &  &  & 1430/1430 & 4862/4862 \\
\hline
(3, 3, 3) &  &  &  &  &  & 4862/4862 \\
\hline
(4, 3, 2) &  &  &  &  &  & \color{orange} 4593/4862\color{black} \\
\hline
(4, 4, 1) &  &  &  &  &  & \color{orange} 4625/4862\color{black} \\
\hline
(5, 2, 2) &  &  &  &  &  & 4862/4862 \\
\hline
(5, 3, 1) &  &  &  &  &  & \color{magenta} 4144/4862\color{black} \\
\hline
(6, 2, 1) &  &  &  &  &  & \color{magenta} 4068/4862\color{black} \\
\hline
(7, 1, 1) &  &  &  &  &  & 4862/4862 \\
\hline
\end{tabular}}

\includegraphics[width=4.5cm]{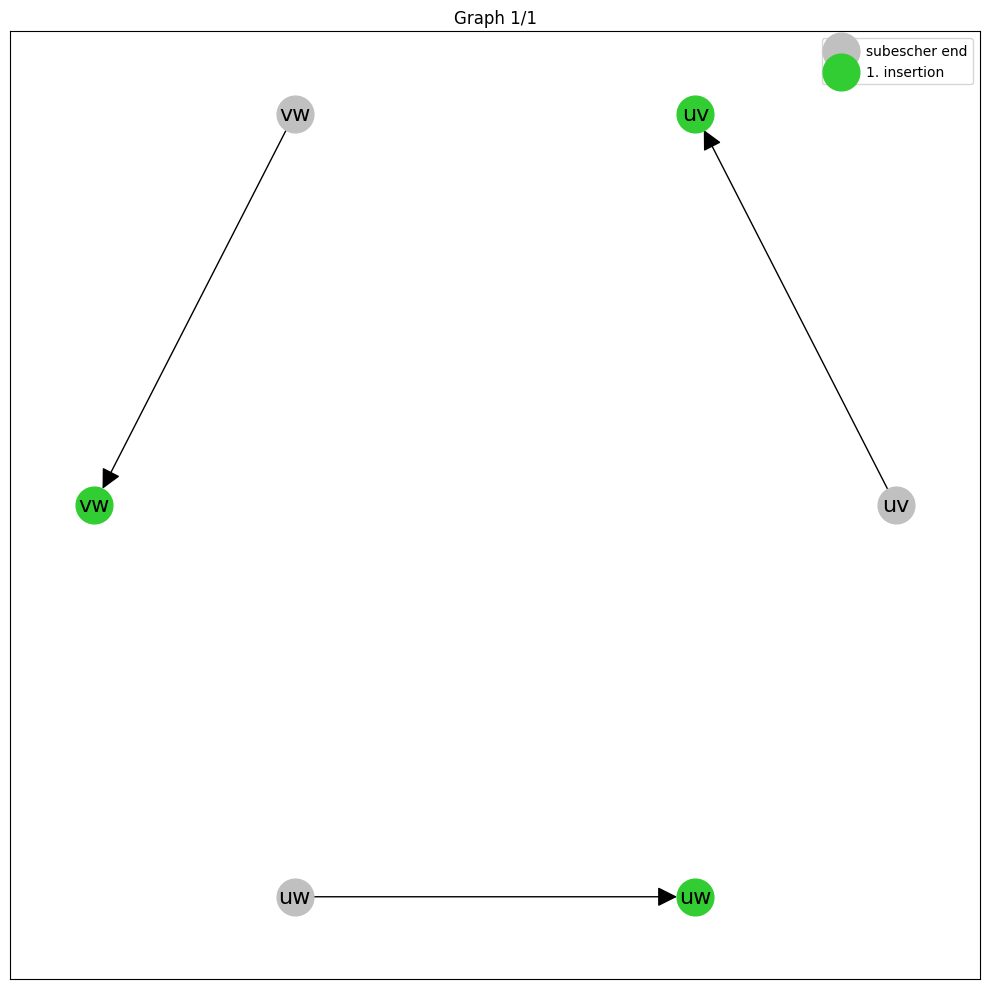}

\caption{Triple partitions: performance of the model [SEEnd(u,v) $<$ FirstIns(u,v)] AND [SEEnd(v,w) $<$ FirstIns(v,w)] AND [SEEnd(u,w) $<$ FirstIns(u,w)].}
\label{tab:triple1}
\end{table}

\clearpage

\begin{table}[h!]
\centering
\resizebox{0.65\textwidth}{!}{
\begin{tabular}{|l||*{6}{c|}}
\hline
\multicolumn{7}{|c|}{Total absolute error/sum of real coefficients for all UIO} \\
\hline
\backslashbox{Partition}{UIO length}
&\makebox[3em]{4}&\makebox[3em]{5}&\makebox[3em]{6}
&\makebox[3em]{7}&\makebox[3em]{8}&\makebox[3em]{9}\\\hline\hline
(2, 1, 1) & \color{cyan} 4/16\color{black} & \color{cyan} 60/256\color{black} & \color{cyan} 564/2552\color{black} & \color{cyan} 4252/20304\color{black} & \color{cyan} 28116/141088\color{black} & \color{cyan} 170364/895072\color{black} \\
\hline
(2, 2, 1) &  & \color{cyan} 10/42\color{black} & \color{cyan} 196/860\color{black} & \color{cyan} 2278/10454\color{black} & \color{cyan} 20456/98088\color{black} & \color{cyan} 156854/784470\color{black} \\
\hline
(3, 1, 1) &  & \color{cyan} 52/106\color{black} & \color{cyan} 960/2052\color{black} & \color{cyan} 10676/23846\color{black} & \color{cyan} 92656/215620\color{black} & \color{cyan} 691364/1671830\color{black} \\
\hline
(2, 2, 2) &  &  & 0/108 & 0/2712 & 0/39216 & 0/427572 \\
\hline
(3, 2, 1) &  &  & \color{cyan} 178/286\color{black} & \color{cyan} 4100/6872\color{black} & \color{cyan} 54796/95536\color{black} & \color{cyan} 556260/1006152\color{black} \\
\hline
(4, 1, 1) &  &  & \color{cyan} 664/1024\color{black} & \color{cyan} 14508/23148\color{black} & \color{cyan} 185808/306112\color{black} & \color{cyan} 1820580/3091580\color{black} \\
\hline
(3, 2, 2) &  &  &  & \color{cyan} 454/742\color{black} & \color{cyan} 12468/21372\color{black} & \color{cyan} 194370/347850\color{black} \\
\hline
(3, 3, 1) &  &  &  & \color{cyan} 834/1824\color{black} & \color{cyan} 22572/50940\color{black} & \color{cyan} 345780/805620\color{black} \\
\hline
(4, 2, 1) &  &  &  & \color{cyan} 2026/3034\color{black} & \color{cyan} 53292/82612\color{black} & \color{cyan} 799206/1279950\color{black} \\
\hline
(5, 1, 1) &  &  &  & \color{cyan} 9322/12472\color{black} & \color{cyan} 234984/322284\color{black} & \color{cyan} 3402364/4778064\color{black} \\
\hline
(3, 3, 2) &  &  &  &  & \color{cyan} 2796/4776\color{black} & \color{cyan} 87360/156156\color{black} \\
\hline
(4, 2, 2) &  &  &  &  & \color{cyan} 4792/8344\color{black} & \color{cyan} 147632/267792\color{black} \\
\hline
(4, 3, 1) &  &  &  &  & \color{cyan} 15518/18758\color{black} & \color{cyan} 475244/587564\color{black} \\
\hline
(5, 2, 1) &  &  &  &  & \color{cyan} 35406/41706\color{black} & \color{cyan} 1052460/1265520\color{black} \\
\hline
(6, 1, 1) &  &  &  &  & \color{cyan} 146124/180144\color{black} & \color{cyan} 4170312/5237856\color{black} \\
\hline
(3, 3, 3) &  &  &  &  &  & 0/28620 \\
\hline
(4, 3, 2) &  &  &  &  &  & \color{cyan} 42764/51836\color{black} \\
\hline
(4, 4, 1) &  &  &  &  &  & \color{cyan} 106218/177434\color{black} \\
\hline
(5, 2, 2) &  &  &  &  &  & \color{cyan} 107670/124470\color{black} \\
\hline
(5, 3, 1) &  &  &  &  &  & \color{cyan} 237488/265838\color{black} \\
\hline
(6, 2, 1) &  &  &  &  &  & \color{cyan} 584256/681792\color{black} \\
\hline
(7, 1, 1) &  &  &  &  &  & \color{cyan} 2550966/2987556\color{black} \\
\hline
\end{tabular}}

\resizebox{0.65\textwidth}{!}{
\begin{tabular}{|l||*{6}{c|}}
\hline
\multicolumn{7}{|c|}{Number of correctly predicted/number of all UIOs} \\
\hline
\backslashbox{Partition}{UIO length}
&\makebox[3em]{4}&\makebox[3em]{5}&\makebox[3em]{6}
&\makebox[3em]{7}&\makebox[3em]{8}&\makebox[3em]{9}\\\hline\hline
(2, 1, 1) & \color{cyan} 12/14\color{black} & \color{cyan} 27/42\color{black} & \color{cyan} 58/132\color{black} & \color{cyan} 121/429\color{black} & \color{cyan} 248/1430\color{black} & \color{cyan} 503/4862\color{black} \\
\hline
(2, 2, 1) &  & \color{cyan} 39/42\color{black} & \color{cyan} 101/132\color{black} & \color{cyan} 242/429\color{black} & \color{cyan} 545/1430\color{black} & \color{cyan} 1174/4862\color{black} \\
\hline
(3, 1, 1) &  & \color{cyan} 33/42\color{black} & \color{cyan} 71/132\color{black} & \color{cyan} 152/429\color{black} & \color{cyan} 321/1430\color{black} & \color{cyan} 673/4862\color{black} \\
\hline
(2, 2, 2) &  &  & 132/132 & 429/429 & 1430/1430 & 4862/4862 \\
\hline
(3, 2, 1) &  &  & \color{cyan} 93/132\color{black} & \color{cyan} 192/429\color{black} & \color{cyan} 375/1430\color{black} & \color{cyan} 715/4862\color{black} \\
\hline
(4, 1, 1) &  &  & \color{cyan} 96/132\color{black} & \color{cyan} 212/429\color{black} & \color{cyan} 473/1430\color{black} & \color{cyan} 1049/4862\color{black} \\
\hline
(3, 2, 2) &  &  &  & \color{cyan} 398/429\color{black} & \color{cyan} 1135/1430\color{black} & \color{cyan} 3054/4862\color{black} \\
\hline
(3, 3, 1) &  &  &  & \color{cyan} 381/429\color{black} & \color{cyan} 1022/1430\color{black} & \color{cyan} 2562/4862\color{black} \\
\hline
(4, 2, 1) &  &  &  & \color{cyan} 261/429\color{black} & \color{cyan} 518/1430\color{black} & \color{cyan} 1028/4862\color{black} \\
\hline
(5, 1, 1) &  &  &  & \color{cyan} 288/429\color{black} & \color{cyan} 671/1430\color{black} & \color{cyan} 1604/4862\color{black} \\
\hline
(3, 3, 2) &  &  &  &  & \color{cyan} 1334/1430\color{black} & \color{cyan} 3902/4862\color{black} \\
\hline
(4, 2, 2) &  &  &  &  & \color{cyan} 1276/1430\color{black} & \color{cyan} 3457/4862\color{black} \\
\hline
(4, 3, 1) &  &  &  &  & \color{cyan} 956/1430\color{black} & \color{cyan} 2131/4862\color{black} \\
\hline
(5, 2, 1) &  &  &  &  & \color{cyan} 794/1430\color{black} & \color{cyan} 1634/4862\color{black} \\
\hline
(6, 1, 1) &  &  &  &  & \color{cyan} 895/1430\color{black} & \color{cyan} 2210/4862\color{black} \\
\hline
(3, 3, 3) &  &  &  &  &  & 4862/4862 \\
\hline
(4, 3, 2) &  &  &  &  &  & \color{cyan} 3948/4862\color{black} \\
\hline
(4, 4, 1) &  &  &  &  &  & \color{cyan} 4122/4862\color{black} \\
\hline
(5, 2, 2) &  &  &  &  &  & \color{cyan} 4160/4862\color{black} \\
\hline
(5, 3, 1) &  &  &  &  &  & \color{cyan} 2839/4862\color{black} \\
\hline
(6, 2, 1) &  &  &  &  &  & \color{cyan} 2556/4862\color{black} \\
\hline
(7, 1, 1) &  &  &  &  &  & \color{cyan} 2894/4862\color{black} \\
\hline
\end{tabular}}

\includegraphics[width=4.5cm]{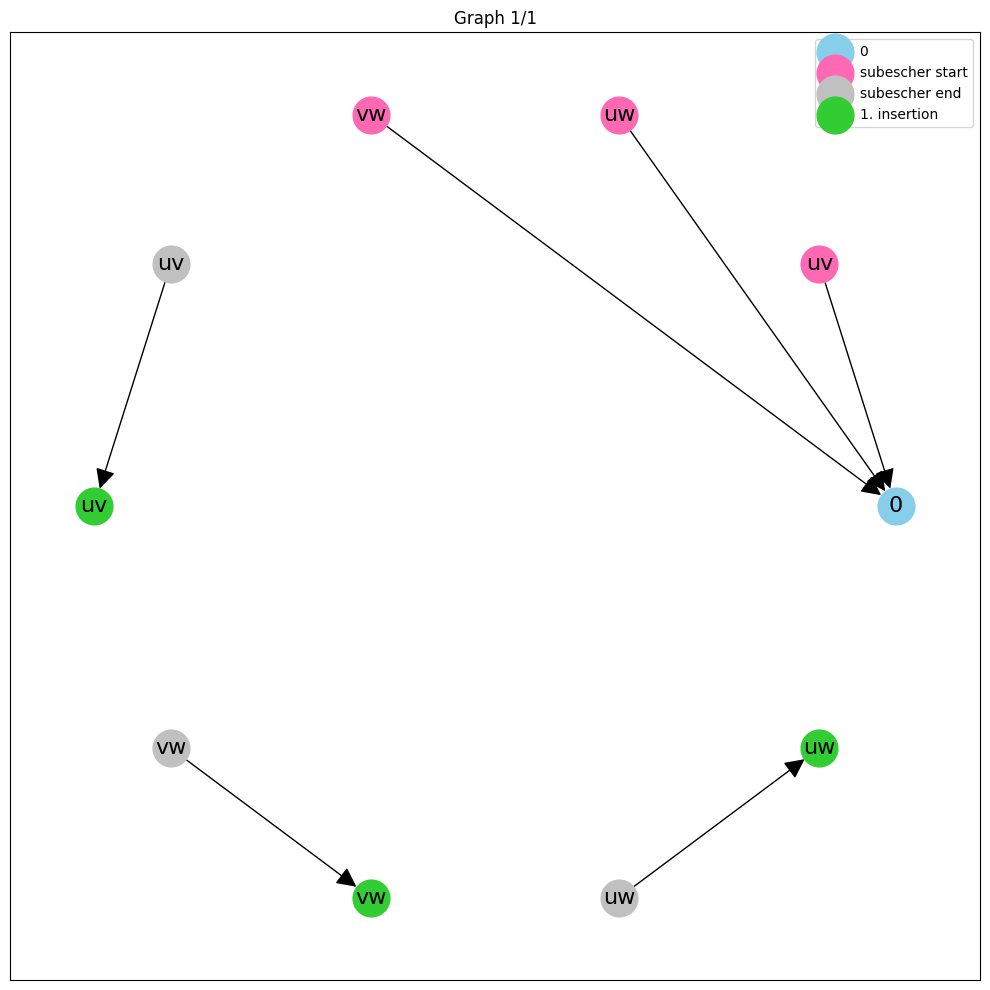}

\caption{Triple partitions with the model [0 $\ge$ SEStart(u,v)] AND [0 $\geq$ SEStart(u,w)] AND [0 $\geq$ SEStart(v,w)] AND [SEEnd(u,v)$<$FirstIns(u,v)] AND [SEEnd(v,w)$<$FirstIns(v,w)] AND [SEEnd(u,w)$<$FirstIns(u,w)], giving lower bound for the Stanley coefficients}
\label{tab:triple2}
\end{table}

\clearpage

\begin{table}[h!]
\centering
\resizebox{0.65\textwidth}{!}{
\begin{tabular}{|l||*{6}{c|}}
\hline
\multicolumn{7}{|c|}{Total absolute error/sum of real coefficients for all UIO} \\
\hline
\backslashbox{Partition}{UIO length}
&\makebox[3em]{4}&\makebox[3em]{5}&\makebox[3em]{6}
&\makebox[3em]{7}&\makebox[3em]{8}&\makebox[3em]{9}\\\hline\hline
(1, 1, 1, 1) & 0/24 & 0/384 & 0/3864 & 0/31224 & 0/221256 & 0/1435056 \\
\hline
(2, 1, 1, 1) &  & 0/66 & 0/1332 & 0/16158 & 0/152556 & 0/1234590 \\
\hline
(2, 2, 1, 1) &  &  & 0/180 & 0/4452 & 0/63976 & 0/698004 \\
\hline
(3, 1, 1, 1) &  &  & 0/456 & 0/10800 & 0/149676 & 0/1583436 \\
\hline
(2, 2, 2, 1) &  &  &  & 0/486 & 0/14316 & 0/239166 \\
\hline
(3, 2, 1, 1) &  &  &  & \color{orange} 28/1272\color{black} & \color{orange} 702/36148\color{black} & \color{orange} 10150/585036\color{black} \\
\hline
(4, 1, 1, 1) &  &  &  & 0/4608 & 0/124716 & 0/1936776 \\
\hline
(2, 2, 2, 2) &  &  &  &  & 0/1296 & 0/44352 \\
\hline
(3, 2, 2, 1) &  &  &  &  & \color{orange} 72/3476\color{black} & \color{orange} 2120/115596\color{black} \\
\hline
(3, 3, 1, 1) &  &  &  &  & \color{orange} 144/8648\color{black} & \color{orange} 4280/279956\color{black} \\
\hline
(4, 2, 1, 1) &  &  &  &  & \color{magenta} 422/13804\color{black} & \color{magenta} 11560/439024\color{black} \\
\hline
(5, 1, 1, 1) &  &  &  &  & 0/58728 & 0/1787232 \\
\hline
(3, 2, 2, 2) &  &  &  &  &  & 0/9306 \\
\hline
(3, 3, 2, 1) &  &  &  &  &  & \color{orange} 732/23864\color{black} \\
\hline
(4, 2, 2, 1) &  &  &  &  &  & \color{magenta} 1058/39650\color{black} \\
\hline
(4, 3, 1, 1) &  &  &  &  &  & \color{orange} 5028/92456\color{black} \\
\hline
(5, 2, 1, 1) &  &  &  &  &  & \color{magenta} 7176/193420\color{black} \\
\hline
(6, 1, 1, 1) &  &  &  &  &  & 0/886716 \\
\hline
\end{tabular}}

\resizebox{0.65\textwidth}{!}{
\begin{tabular}{|l||*{6}{c|}}
\hline
\multicolumn{7}{|c|}{Number of correctly predicted/number of all UIOs} \\
\hline
\backslashbox{Partition}{UIO length}
&\makebox[3em]{4}&\makebox[3em]{5}&\makebox[3em]{6}
&\makebox[3em]{7}&\makebox[3em]{8}&\makebox[3em]{9}\\\hline\hline
(1, 1, 1, 1) & 14/14 & 42/42 & 132/132 & 429/429 & 1430/1430 & 4862/4862 \\
\hline
(2, 1, 1, 1) &  & 42/42 & 132/132 & 429/429 & 1430/1430 & 4862/4862 \\
\hline
(2, 2, 1, 1) &  &  & 132/132 & 429/429 & 1430/1430 & 4862/4862 \\
\hline
(3, 1, 1, 1) &  &  & 132/132 & 429/429 & 1430/1430 & 4862/4862 \\
\hline
(2, 2, 2, 1) &  &  &  & 429/429 & 1430/1430 & 4862/4862 \\
\hline
(3, 2, 1, 1) &  &  &  & \color{orange} 420/429\color{black} & \color{orange} 1306/1430\color{black} & \color{orange} 3920/4862\color{black} \\
\hline
(4, 1, 1, 1) &  &  &  & 429/429 & 1430/1430 & 4862/4862 \\
\hline
(2, 2, 2, 2) &  &  &  &  & 1430/1430 & 4862/4862 \\
\hline
(3, 2, 2, 1) &  &  &  &  & \color{orange} 1415/1430\color{black} & \color{orange} 4620/4862\color{black} \\
\hline
(3, 3, 1, 1) &  &  &  &  & \color{orange} 1414/1430\color{black} & \color{orange} 4610/4862\color{black} \\
\hline
(4, 2, 1, 1) &  &  &  &  & \color{magenta} 1347/1430\color{black} & \color{magenta} 4054/4862\color{black} \\
\hline
(5, 1, 1, 1) &  &  &  &  & 1430/1430 & 4862/4862 \\
\hline
(3, 2, 2, 2) &  &  &  &  &  & 4862/4862 \\
\hline
(3, 3, 2, 1) &  &  &  &  &  & \color{orange} 4762/4862\color{black} \\
\hline
(4, 2, 2, 1) &  &  &  &  &  & \color{magenta} 4730/4862\color{black} \\
\hline
(4, 3, 1, 1) &  &  &  &  &  & \color{orange} 4388/4862\color{black} \\
\hline
(5, 2, 1, 1) &  &  &  &  &  & \color{magenta} 4290/4862\color{black} \\
\hline
(6, 1, 1, 1) &  &  &  &  &  & 4862/4862 \\
\hline
\end{tabular}}

\includegraphics[width=5cm]{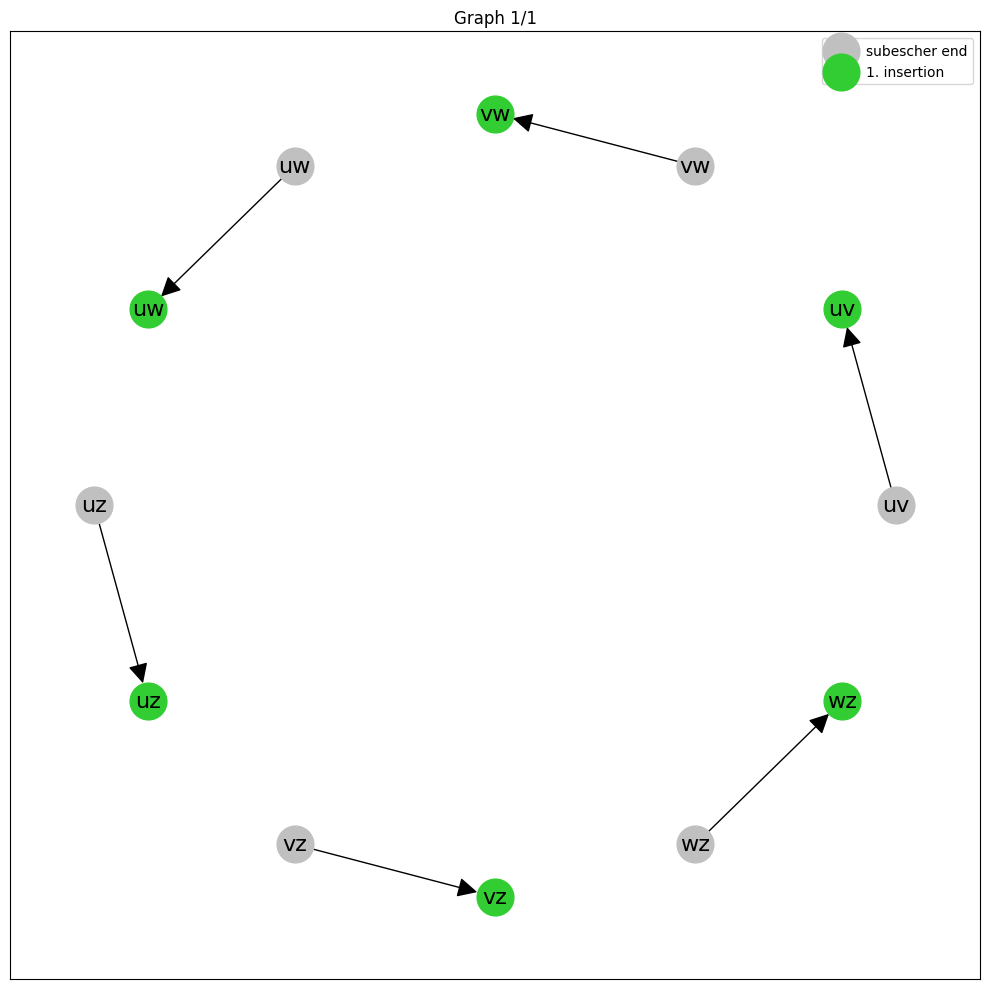}

\caption{Quadruple partition with the canonical model}
\label{tab:quadruple1}
\end{table}

\clearpage

\begin{table}[h!]
\centering
\resizebox{0.65\textwidth}{!}{
\begin{tabular}{|l||*{6}{c|}}
\hline
\multicolumn{7}{|c|}{Total absolute error/sum of real coefficients for all UIO} \\
\hline
\backslashbox{Partition}{UIO length}
&\makebox[3em]{4}&\makebox[3em]{5}&\makebox[3em]{6}
&\makebox[3em]{7}&\makebox[3em]{8}&\makebox[3em]{9}\\\hline\hline
(1, 1, 1, 1) & 0/24 & 0/384 & 0/3864 & 0/31224 & 0/221256 & 0/1435056 \\
\hline
(2, 1, 1, 1) &  & 0/66 & 0/1332 & 0/16158 & 0/152556 & 0/1234590 \\
\hline
(2, 2, 1, 1) &  &  & \color{cyan} 32/180\color{black} & \color{cyan} 768/4452\color{black} & \color{cyan} 10716/63976\color{black} & \color{cyan} 113612/698004\color{black} \\
\hline
(3, 1, 1, 1) &  &  & 0/456 & 0/10800 & 0/149676 & 0/1583436 \\
\hline
(2, 2, 2, 1) &  &  &  & \color{cyan} 84/486\color{black} & \color{cyan} 2424/14316\color{black} & \color{cyan} 39604/239166\color{black} \\
\hline
(3, 2, 1, 1) &  &  &  & \color{cyan} 224/1272\color{black} & \color{cyan} 6196/36148\color{black} & \color{cyan} 97608/585036\color{black} \\
\hline
(4, 1, 1, 1) &  &  &  & 0/4608 & 0/124716 & 0/1936776 \\
\hline
(2, 2, 2, 2) &  &  &  &  & 0/1296 & 0/44352 \\
\hline
(3, 2, 2, 1) &  &  &  &  & \color{cyan} 600/3476\color{black} & \color{cyan} 19534/115596\color{black} \\
\hline
(3, 3, 1, 1) &  &  &  &  & \color{cyan} 3086/8648\color{black} & \color{cyan} 97728/279956\color{black} \\
\hline
(4, 2, 1, 1) &  &  &  &  & \color{cyan} 2480/13804\color{black} & \color{cyan} 76656/439024\color{black} \\
\hline
(5, 1, 1, 1) &  &  &  &  & 0/58728 & 0/1787232 \\
\hline
(3, 2, 2, 2) &  &  &  &  &  & 0/9306 \\
\hline
(3, 3, 2, 1) &  &  &  &  &  & \color{cyan} 11384/23864\color{black} \\
\hline
(4, 2, 2, 1) &  &  &  &  &  & \color{cyan} 7074/39650\color{black} \\
\hline
(4, 3, 1, 1) &  &  &  &  &  & \color{cyan} 32014/92456\color{black} \\
\hline
(5, 2, 1, 1) &  &  &  &  &  & \color{cyan} 35868/193420\color{black} \\
\hline
(6, 1, 1, 1) &  &  &  &  &  & 0/886716 \\
\hline
\end{tabular}}

\resizebox{0.65\textwidth}{!}{
\begin{tabular}{|l||*{6}{c|}}
\hline
\multicolumn{7}{|c|}{Number of correctly predicted/number of all UIOs} \\
\hline
\backslashbox{Partition}{UIO length}
&\makebox[3em]{4}&\makebox[3em]{5}&\makebox[3em]{6}
&\makebox[3em]{7}&\makebox[3em]{8}&\makebox[3em]{9}\\\hline\hline
(1, 1, 1, 1) & 14/14 & 42/42 & 132/132 & 429/429 & 1430/1430 & 4862/4862 \\
\hline
(2, 1, 1, 1) &  & 42/42 & 132/132 & 429/429 & 1430/1430 & 4862/4862 \\
\hline
(2, 2, 1, 1) &  &  & \color{cyan} 123/132\color{black} & \color{cyan} 338/429\color{black} & \color{cyan} 882/1430\color{black} & \color{cyan} 2232/4862\color{black} \\
\hline
(3, 1, 1, 1) &  &  & 132/132 & 429/429 & 1430/1430 & 4862/4862 \\
\hline
(2, 2, 2, 1) &  &  &  & \color{cyan} 421/429\color{black} & \color{cyan} 1314/1430\color{black} & \color{cyan} 3934/4862\color{black} \\
\hline
(3, 2, 1, 1) &  &  &  & \color{cyan} 393/429\color{black} & \color{cyan} 1100/1430\color{black} & \color{cyan} 2927/4862\color{black} \\
\hline
(4, 1, 1, 1) &  &  &  & 429/429 & 1430/1430 & 4862/4862 \\
\hline
(2, 2, 2, 2) &  &  &  &  & 1430/1430 & 4862/4862 \\
\hline
(3, 2, 2, 1) &  &  &  &  & \color{cyan} 1371/1430\color{black} & \color{cyan} 4172/4862\color{black} \\
\hline
(3, 3, 1, 1) &  &  &  &  & \color{cyan} 1282/1430\color{black} & \color{cyan} 3497/4862\color{black} \\
\hline
(4, 2, 1, 1) &  &  &  &  & \color{cyan} 1238/1430\color{black} & \color{cyan} 3308/4862\color{black} \\
\hline
(5, 1, 1, 1) &  &  &  &  & 1430/1430 & 4862/4862 \\
\hline
(3, 2, 2, 2) &  &  &  &  &  & 4862/4862 \\
\hline
(3, 3, 2, 1) &  &  &  &  &  & \color{cyan} 4287/4862\color{black} \\
\hline
(4, 2, 2, 1) &  &  &  &  &  & \color{cyan} 4531/4862\color{black} \\
\hline
(4, 3, 1, 1) &  &  &  &  &  & \color{cyan} 4150/4862\color{black} \\
\hline
(5, 2, 1, 1) &  &  &  &  &  & \color{cyan} 3955/4862\color{black} \\
\hline
(6, 1, 1, 1) &  &  &  &  &  & 4862/4862 \\
\hline
\end{tabular}}

\includegraphics[width=5cm]{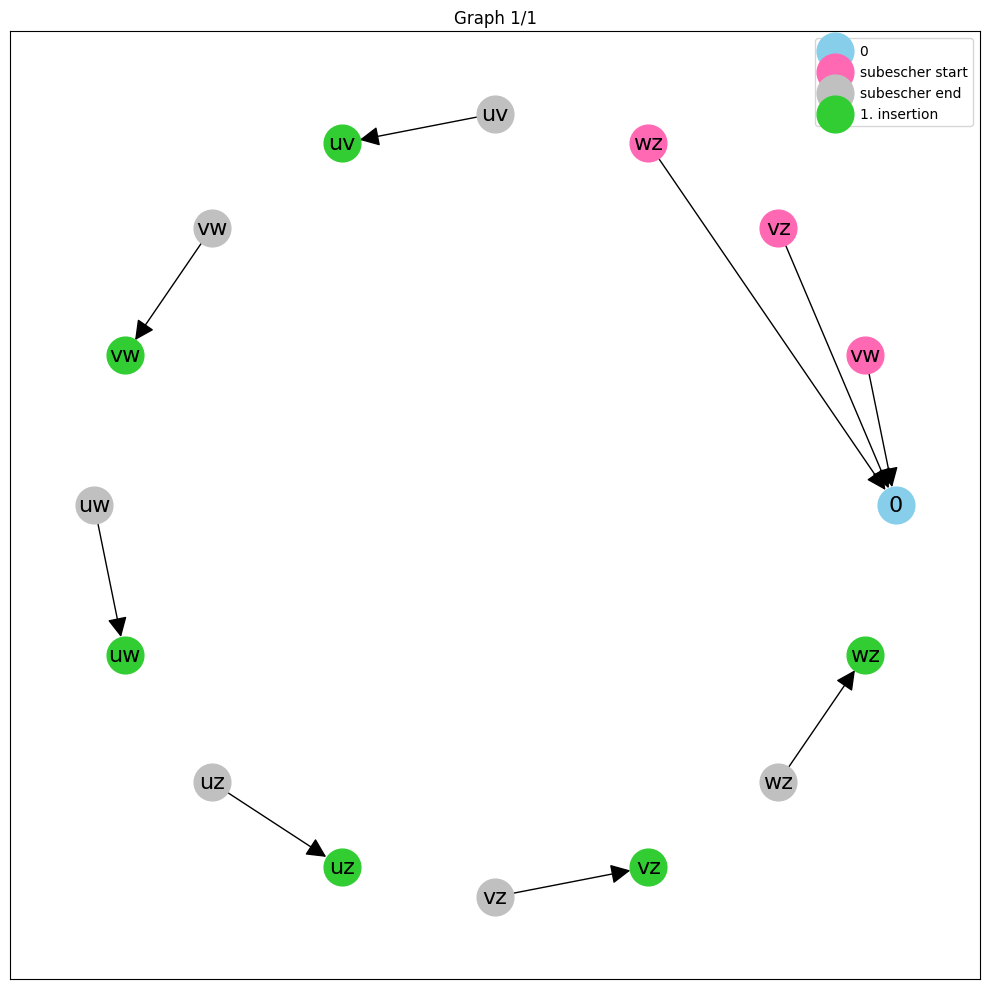}

\caption{Quadruple partitions with the modified canonincal model, giving lower bound for the Stanley coefficients. Also the best conditions when allowing up to 6 additional conditions while only using the uv,uw,uz,vw,vz,wz subescher for the [0 $\geq$ subescher start] conditions}
\label{tab:quadruple2}
\end{table}

\clearpage

\begin{table}[h!]
\centering
\resizebox{0.65\textwidth}{!}{
\begin{tabular}{|l||*{6}{c|}}
\hline
\multicolumn{7}{|c|}{Total absolute error/sum of real coefficients for all UIO} \\
\hline
\backslashbox{Partition}{UIO length}
&\makebox[3em]{4}&\makebox[3em]{5}&\makebox[3em]{6}
&\makebox[3em]{7}&\makebox[3em]{8}&\makebox[3em]{9}\\\hline\hline
(2, 1, 1) & 0/16 & 0/256 & 0/2552 & 0/20304 & 0/141088 & 0/895072 \\
\hline
(2, 2, 1) &  & 0/42 & 0/860 & 0/10454 & 0/98088 & 0/784470 \\
\hline
(3, 1, 1) &  & 0/106 & 0/2052 & 0/23846 & 0/215620 & 0/1671830 \\
\hline
(2, 2, 2) &  &  & 0/108 & 0/2712 & 0/39216 & 0/427572 \\
\hline
(3, 2, 1) &  &  & \color{orange} 4/286\color{black} & \color{orange} 80/6872\color{black} & \color{orange} 948/95536\color{black} & \color{orange} 8668/1006152\color{black} \\
\hline
(4, 1, 1) &  &  & 0/1024 & 0/23148 & 0/306112 & 0/3091580 \\
\hline
(3, 2, 2) &  &  &  & 0/742 & 0/21372 & 0/347850 \\
\hline
(3, 3, 1) &  &  &  & \color{orange} 16/1824\color{black} & \color{orange} 400/50940\color{black} & \color{orange} 5736/805620\color{black} \\
\hline
(4, 2, 1) &  &  &  & \color{magenta} 48/3034\color{black} & \color{magenta} 963/82612\color{black} & \color{magenta} 11736/1279950\color{black} \\
\hline
(5, 1, 1) &  &  &  & 0/12472 & 0/322284 & 0/4778064 \\
\hline
(3, 3, 2) &  &  &  &  & \color{orange} 90/4776\color{black} & \color{orange} 2616/156156\color{black} \\
\hline
(4, 2, 2) &  &  &  &  & 0/8344 & 0/267792 \\
\hline
(4, 3, 1) &  &  &  &  & \color{magenta} 644/18758\color{black} & \color{magenta} 17877/587564\color{black} \\
\hline
(5, 2, 1) &  &  &  &  & \color{magenta} 746/41706\color{black} & \color{magenta} 15840/1265520\color{black} \\
\hline
(6, 1, 1) &  &  &  &  & 0/180144 & 0/5237856 \\
\hline
(3, 3, 3) &  &  &  &  &  & 0/28620 \\
\hline
(4, 3, 2) &  &  &  &  &  & \color{magenta} 1380/51836\color{black} \\
\hline
(4, 4, 1) &  &  &  &  &  & \color{orange} 3236/177434\color{black} \\
\hline
(5, 2, 2) &  &  &  &  &  & 0/124470 \\
\hline
(5, 3, 1) &  &  &  &  &  & \color{magenta} 9725/265838\color{black} \\
\hline
(6, 2, 1) &  &  &  &  &  & \color{magenta} 13066/681792\color{black} \\
\hline
(7, 1, 1) &  &  &  &  &  & 0/2987556 \\
\hline
\end{tabular}}

\resizebox{0.65\textwidth}{!}{
\begin{tabular}{|l||*{6}{c|}}
\hline
\multicolumn{7}{|c|}{Number of correctly predicted/number of all UIOs} \\
\hline
\backslashbox{Partition}{UIO length}
&\makebox[3em]{4}&\makebox[3em]{5}&\makebox[3em]{6}
&\makebox[3em]{7}&\makebox[3em]{8}&\makebox[3em]{9}\\\hline\hline
(2, 1, 1) & 14/14 & 42/42 & 132/132 & 429/429 & 1430/1430 & 4862/4862 \\
\hline
(2, 2, 1) &  & 42/42 & 132/132 & 429/429 & 1430/1430 & 4862/4862 \\
\hline
(3, 1, 1) &  & 42/42 & 132/132 & 429/429 & 1430/1430 & 4862/4862 \\
\hline
(2, 2, 2) &  &  & 132/132 & 429/429 & 1430/1430 & 4862/4862 \\
\hline
(3, 2, 1) &  &  & \color{orange} 130/132\color{black} & \color{orange} 401/429\color{black} & \color{orange} 1209/1430\color{black} & \color{orange} 3566/4862\color{black} \\
\hline
(4, 1, 1) &  &  & 132/132 & 429/429 & 1430/1430 & 4862/4862 \\
\hline
(3, 2, 2) &  &  &  & 429/429 & 1430/1430 & 4862/4862 \\
\hline
(3, 3, 1) &  &  &  & \color{orange} 425/429\color{black} & \color{orange} 1369/1430\color{black} & \color{orange} 4344/4862\color{black} \\
\hline
(4, 2, 1) &  &  &  & \color{magenta} 411/429\color{black} & \color{magenta} 1248/1430\color{black} & \color{magenta} 3679/4862\color{black} \\
\hline
(5, 1, 1) &  &  &  & 429/429 & 1430/1430 & 4862/4862 \\
\hline
(3, 3, 2) &  &  &  &  & \color{orange} 1406/1430\color{black} & \color{orange} 4526/4862\color{black} \\
\hline
(4, 2, 2) &  &  &  &  & 1430/1430 & 4862/4862 \\
\hline
(4, 3, 1) &  &  &  &  & \color{magenta} 1317/1430\color{black} & \color{magenta} 3886/4862\color{black} \\
\hline
(5, 2, 1) &  &  &  &  & \color{magenta} 1306/1430\color{black} & \color{magenta} 3848/4862\color{black} \\
\hline
(6, 1, 1) &  &  &  &  & 1430/1430 & 4862/4862 \\
\hline
(3, 3, 3) &  &  &  &  &  & 4862/4862 \\
\hline
(4, 3, 2) &  &  &  &  &  & \color{magenta} 4588/4862\color{black} \\
\hline
(4, 4, 1) &  &  &  &  &  & \color{orange} 4625/4862\color{black} \\
\hline
(5, 2, 2) &  &  &  &  &  & 4862/4862 \\
\hline
(5, 3, 1) &  &  &  &  &  & \color{magenta} 4193/4862\color{black} \\
\hline
(6, 2, 1) &  &  &  &  &  & \color{magenta} 4164/4862\color{black} \\
\hline
(7, 1, 1) &  &  &  &  &  & 4862/4862 \\
\hline
\end{tabular}}

\includegraphics[width=4.5cm]{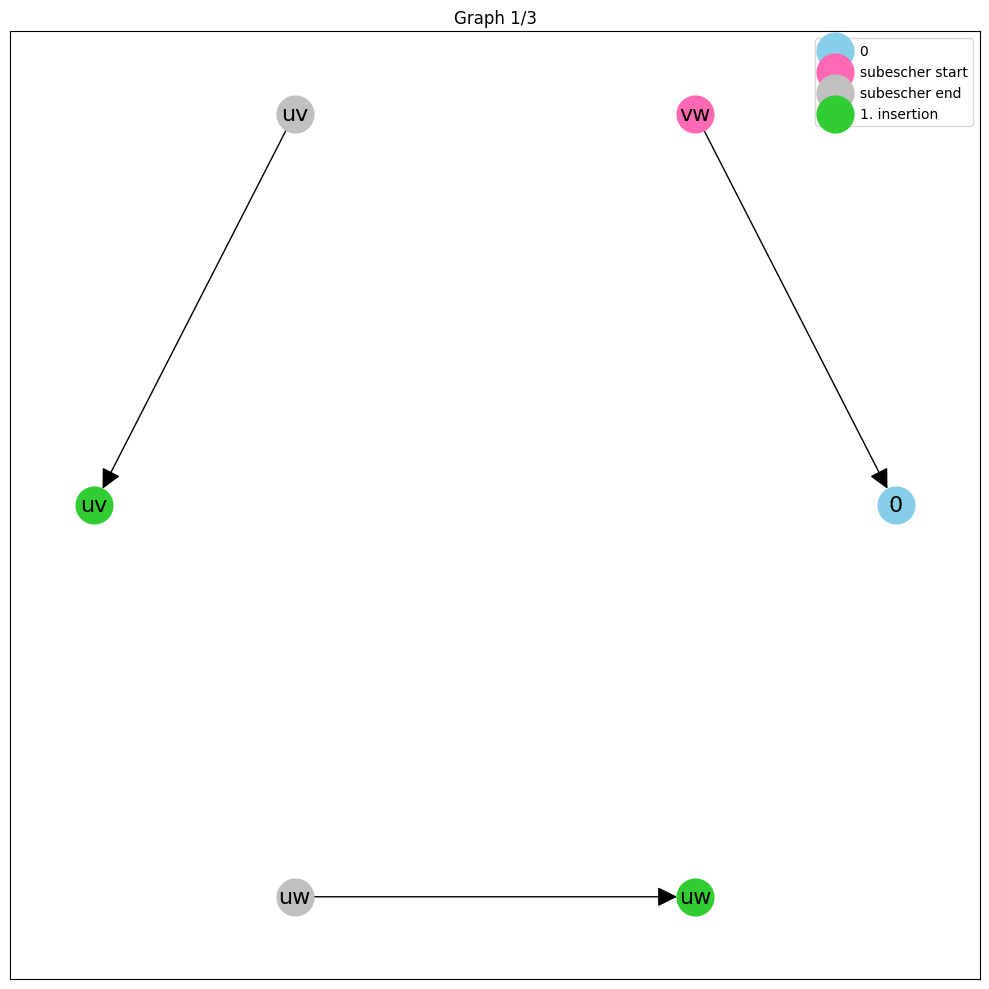}
\includegraphics[width=4.5cm]{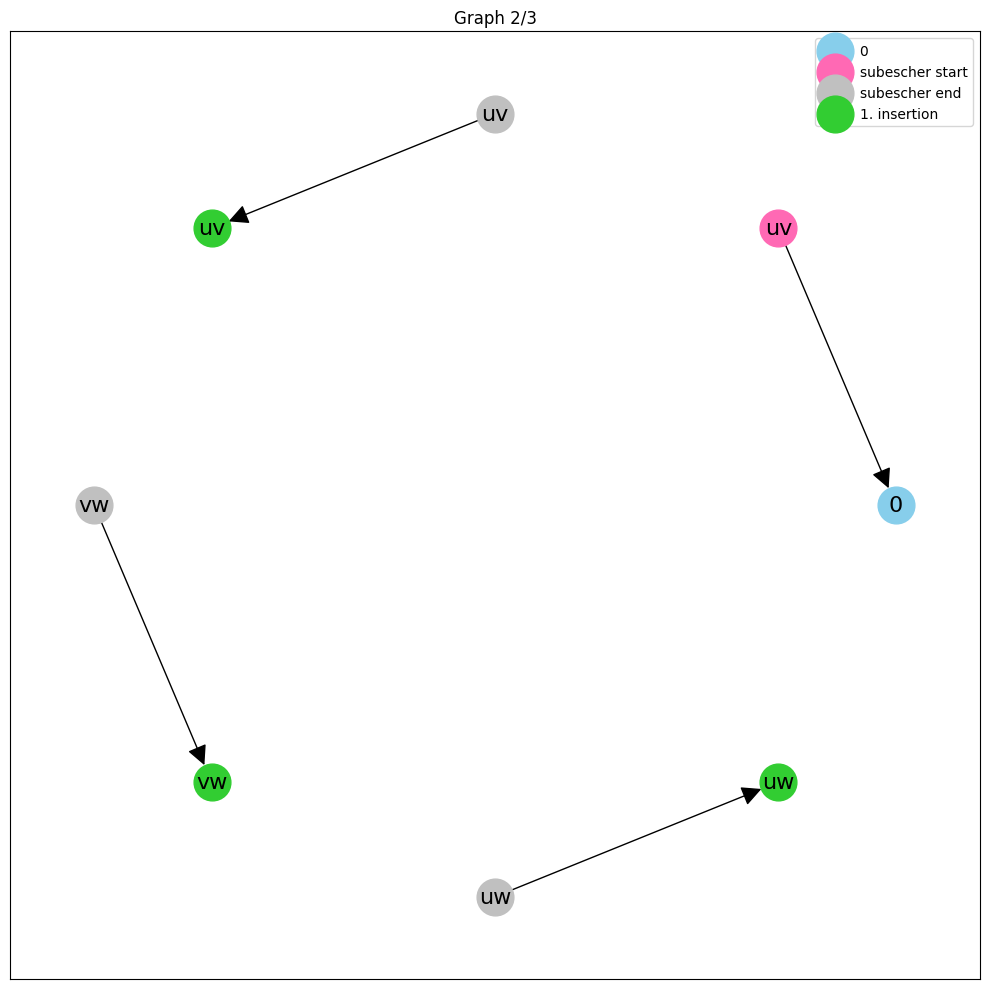}
\includegraphics[width=4.5cm]{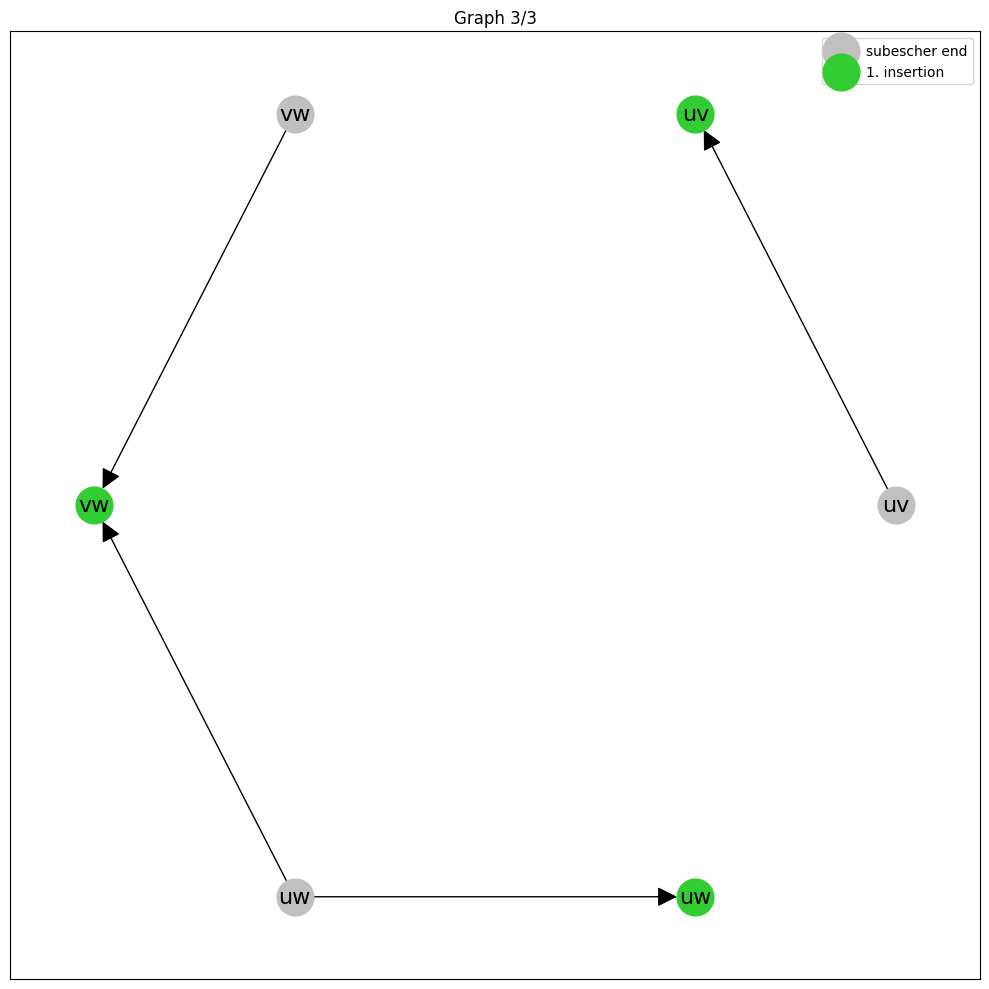}

\caption{A model with 3 condition graphs meaning the Boolean expression [graph1 OR graph2 OR graph3]. It is trained on $(5,2,1)$ partition with 9 intervals, and shows that the RL agent often finds better conditions for individual partitions.}
\label{tab:triple3rows}
\end{table}

\clearpage

\begin{table}[h!]
\centering
\resizebox{0.65\textwidth}{!}{
\begin{tabular}{|l||*{6}{c|}}
\hline
\multicolumn{7}{|c|}{Total absolute error/sum of real coefficients for all UIO} \\
\hline
\backslashbox{Partition}{UIO length}
&\makebox[3em]{4}&\makebox[3em]{5}&\makebox[3em]{6}
&\makebox[3em]{7}&\makebox[3em]{8}&\makebox[3em]{9}\\\hline\hline
(1, 1, 1, 1) & 0/24 & 0/384 & 0/3864 & 0/31224 & 0/221256 & 0/1435056 \\
\hline
(2, 1, 1, 1) &  & 0/66 & 0/1332 & 0/16158 & 0/152556 & 0/1234590 \\
\hline
(2, 2, 1, 1) &  &  & \color{cyan} 32/180\color{black} & \color{cyan} 768/4452\color{black} & \color{cyan} 10716/63976\color{black} & \color{cyan} 113612/698004\color{black} \\
\hline
(3, 1, 1, 1) &  &  & 0/456 & 0/10800 & 0/149676 & 0/1583436 \\
\hline
(2, 2, 2, 1) &  &  &  & \color{cyan} 84/486\color{black} & \color{cyan} 2424/14316\color{black} & \color{cyan} 39604/239166\color{black} \\
\hline
(3, 2, 1, 1) &  &  &  & \color{cyan} 224/1272\color{black} & \color{cyan} 6196/36148\color{black} & \color{cyan} 97608/585036\color{black} \\
\hline
(4, 1, 1, 1) &  &  &  & 0/4608 & 0/124716 & 0/1936776 \\
\hline
(2, 2, 2, 2) &  &  &  &  & 0/1296 & 0/44352 \\
\hline
(3, 2, 2, 1) &  &  &  &  & \color{cyan} 600/3476\color{black} & \color{cyan} 19534/115596\color{black} \\
\hline
(3, 3, 1, 1) &  &  &  &  & \color{cyan} 3086/8648\color{black} & \color{cyan} 97728/279956\color{black} \\
\hline
(4, 2, 1, 1) &  &  &  &  & \color{cyan} 2480/13804\color{black} & \color{cyan} 76656/439024\color{black} \\
\hline
(5, 1, 1, 1) &  &  &  &  & 0/58728 & 0/1787232 \\
\hline
(3, 2, 2, 2) &  &  &  &  &  & 0/9306 \\
\hline
(3, 3, 2, 1) &  &  &  &  &  & \color{cyan} 11384/23864\color{black} \\
\hline
(4, 2, 2, 1) &  &  &  &  &  & \color{cyan} 7074/39650\color{black} \\
\hline
(4, 3, 1, 1) &  &  &  &  &  & \color{cyan} 32014/92456\color{black} \\
\hline
(5, 2, 1, 1) &  &  &  &  &  & \color{cyan} 35868/193420\color{black} \\
\hline
(6, 1, 1, 1) &  &  &  &  &  & 0/886716 \\
\hline
\end{tabular}}

\resizebox{0.65\textwidth}{!}{
\begin{tabular}{|l||*{6}{c|}}
\hline
\multicolumn{7}{|c|}{Number of correctly predicted/number of all UIOs} \\
\hline
\backslashbox{Partition}{UIO length}
&\makebox[3em]{4}&\makebox[3em]{5}&\makebox[3em]{6}
&\makebox[3em]{7}&\makebox[3em]{8}&\makebox[3em]{9}\\\hline\hline
(1, 1, 1, 1) & 14/14 & 42/42 & 132/132 & 429/429 & 1430/1430 & 4862/4862 \\
\hline
(2, 1, 1, 1) &  & 42/42 & 132/132 & 429/429 & 1430/1430 & 4862/4862 \\
\hline
(2, 2, 1, 1) &  &  & \color{cyan} 123/132\color{black} & \color{cyan} 338/429\color{black} & \color{cyan} 882/1430\color{black} & \color{cyan} 2232/4862\color{black} \\
\hline
(3, 1, 1, 1) &  &  & 132/132 & 429/429 & 1430/1430 & 4862/4862 \\
\hline
(2, 2, 2, 1) &  &  &  & \color{cyan} 421/429\color{black} & \color{cyan} 1314/1430\color{black} & \color{cyan} 3934/4862\color{black} \\
\hline
(3, 2, 1, 1) &  &  &  & \color{cyan} 393/429\color{black} & \color{cyan} 1100/1430\color{black} & \color{cyan} 2927/4862\color{black} \\
\hline
(4, 1, 1, 1) &  &  &  & 429/429 & 1430/1430 & 4862/4862 \\
\hline
(2, 2, 2, 2) &  &  &  &  & 1430/1430 & 4862/4862 \\
\hline
(3, 2, 2, 1) &  &  &  &  & \color{cyan} 1371/1430\color{black} & \color{cyan} 4172/4862\color{black} \\
\hline
(3, 3, 1, 1) &  &  &  &  & \color{cyan} 1282/1430\color{black} & \color{cyan} 3497/4862\color{black} \\
\hline
(4, 2, 1, 1) &  &  &  &  & \color{cyan} 1238/1430\color{black} & \color{cyan} 3308/4862\color{black} \\
\hline
(5, 1, 1, 1) &  &  &  &  & 1430/1430 & 4862/4862 \\
\hline
(3, 2, 2, 2) &  &  &  &  &  & 4862/4862 \\
\hline
(3, 3, 2, 1) &  &  &  &  &  & \color{cyan} 4287/4862\color{black} \\
\hline
(4, 2, 2, 1) &  &  &  &  &  & \color{cyan} 4531/4862\color{black} \\
\hline
(4, 3, 1, 1) &  &  &  &  &  & \color{cyan} 4150/4862\color{black} \\
\hline
(5, 2, 1, 1) &  &  &  &  &  & \color{cyan} 3955/4862\color{black} \\
\hline
(6, 1, 1, 1) &  &  &  &  &  & 4862/4862 \\
\hline
\end{tabular}}

\includegraphics[width=5cm]{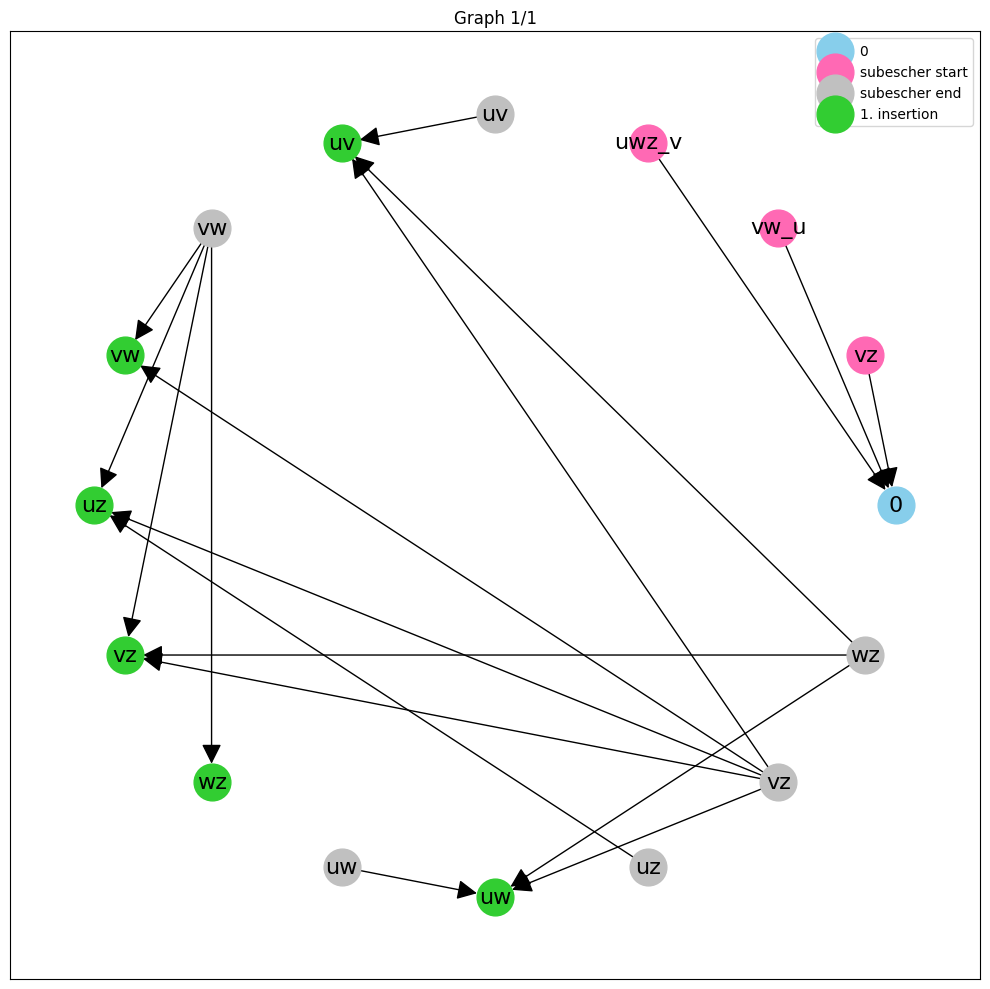}

\caption{A model found by the RL agent for quadruple partitions with only negative errors, and with better performance than the model of Theorem \ref{mainthm3}. This shows that graphs which give lower bound for the Stanley coefficients are not unique, but finding one with general pattern is harder.}
\label{tab:quadruplecrazy}
\end{table}

\clearpage

\section{Analysis of the Machine Learning Method}

The machine learning approach we employed to tackle this problem presented several notable challenges. Below, we outline key observations and strategies that emerged during the process, along with the implications for optimizing the training process.

\subsection{Non-Uniqueness of Solutions and Balancing Complexity}

For most partition types, there is no single best solution. Even when a perfect solution (i.e., a solution that achieves a zero-error score) is found, the corresponding condition graph is often not unique. Table \ref{tab:triple1} and \ref{tab:triple3rows} illustrate this for the case of the partition $(a, b, b)$. This observation indicates the need to balance between obtaining optimal solutions with potentially highly complex condition graphs and accepting slightly suboptimal solutions that correspond to significantly simpler, more interpretable, and potentially more universal condition graphs. This trade-off between accuracy and simplicity is critical to the applicability and generalization of our model.

\subsection{Managing the Comparison Space}

As the size of the partition grows, the comparison space becomes exponentially large, leading to impractical memory requirements during training. For example, training the model for partitions of length greater than five quickly exhausts available resources. To mitigate this, we must reduce the comparison space, which necessitates leveraging mathematical intuition. By identifying symmetries and reducing redundant comparisons, we can make the search space more tractable while still preserving solution quality.

\subsection{Training Pitfalls in Large Search Spaces}

One of the common pitfalls in training with a large comparison space is the lack of satisfying solutions for random Boolean expressions. In many cases, none of the Escher tuples meet the given conditions after the initial random iteration, causing the model to become trapped in a suboptimal region where the best achievable score is merely the sum of the nonzero Stanley coefficients across all UIOs. To avoid this, we must guide the training process to prevent such stagnation, ensuring that the model does not get trapped in these non-productive domains.

\subsection{Reducing Cores and Comparisons}

An essential aspect of optimizing the model involves reducing the number of cores and comparisons. For example, in the case of $n=3$, the optimal solution only utilizes specific cores (e.g., $SEEnd$ and $FirstIns$), while others, such as $0$ and $SEStart$, are not used. This insight suggests that for $n=4$, we can eliminate the $SEStart$ cores and restrict comparisons to pairs of cores such as $(0, SEStart)$ and $(SEEnd, FirstIns)$. This reduction in action space helps prevent the model from being stuck in local minima during training, significantly improving efficiency.

\subsection{Core Selection and Encoding Decision Trees}

The process of selecting the appropriate cores is critical to model performance. Encoding decision trees with graphs is a complex task, and it took several months to identify the correct core vectors. Specifically, interpreting the meaning of expressions like $\mathbf{SEEnd}(u,v) < \mathbf{FirstIns}(u,v)$ requires a deep understanding of multiple aspects of the graph's structure. These core vectors play a crucial role in ensuring that the model captures the right relationships between the nodes and edges of the decision tree.

\subsection{Condition Graph Row Selection}

Another critical factor in the model's design is the number of rows used in the condition graph. In theory, adding a second row expands the search space to include solutions that require only one row. However, in practice, increasing the number of rows dramatically increases the state space, which, in turn, heightens the risk of the model becoming stuck in good-but-not-optimal solutions. Through experimentation, we found that solutions using 1-3 rows often outperformed those with more rows, even though a larger state space theoretically allows for more complex solutions.

\subsection{Importance of the Score Function}

The choice of score function is pivotal in guiding the reinforcement learning (RL) agent toward optimal solutions. In particular, the role of the edge penalty is crucial. We observed that when $\mathrm{EdgePenalty} > 1$, the RL agent tends to prioritize shorter solutions, i.e., smaller condition graphs, at the expense of better absolute error. Through extensive experimentation, we determined that the optimal range for EdgePenalty lies between 0.1 and 0.5. This range strikes a balance between solution length and accuracy, leading to more efficient exploration of the solution space.

\subsection{Training charts}

We added a ResultViewer folder to the GitHub repository, where the reinforcement training process is monitored and saved as charts. An example is shown in Figure \ref{fig:charts}. 

\clearpage

\begin{figure}
    \centering
    \includegraphics[width=0.4\linewidth]{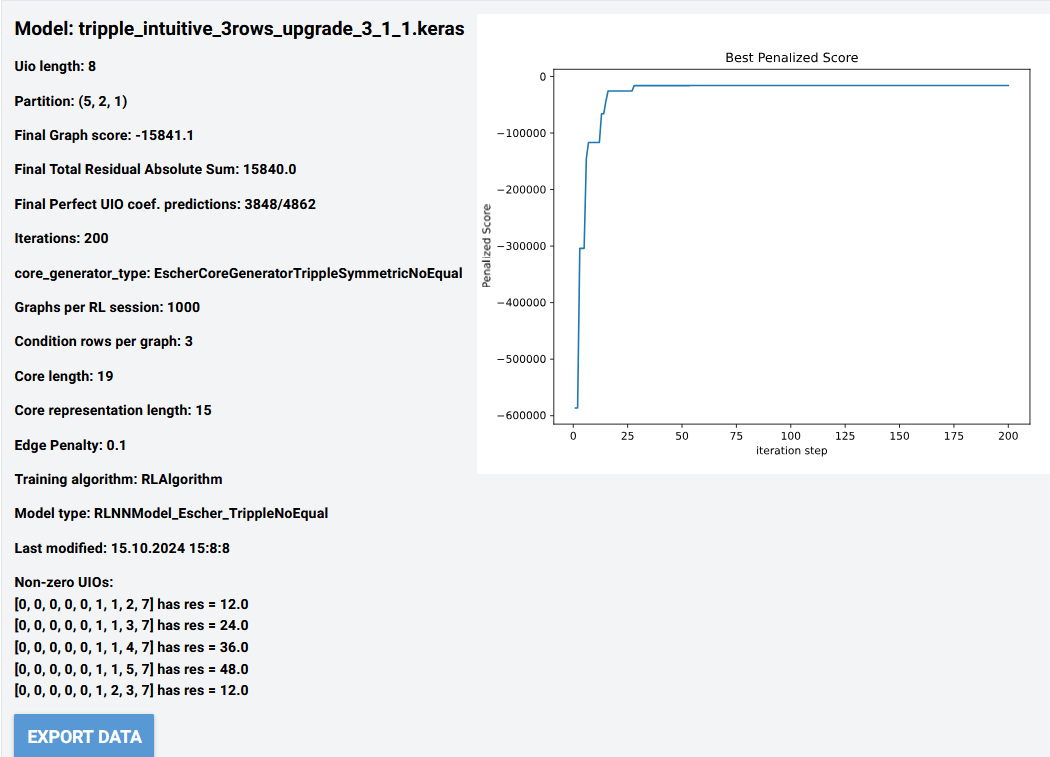}
    \includegraphics[width=0.4\linewidth]{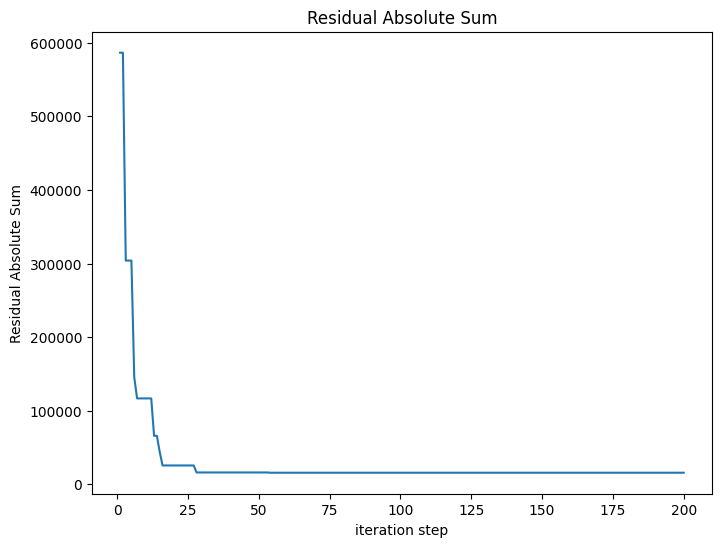}
    \includegraphics[width=0.4\linewidth]{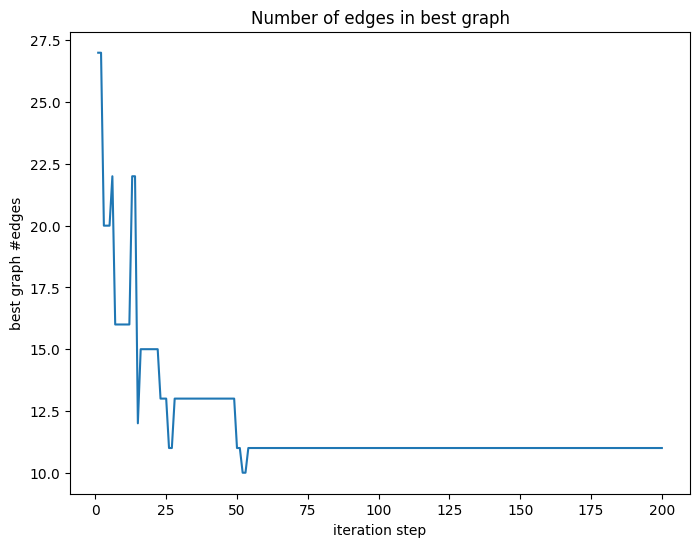}
    \caption{Training charts for the partition $(5,2,1)$ trained on all UIOs of length $8$, with 3 rows in the condition graph. The number of correctly predicted coefficients is $3848/4862$. Because EdgePenalty<1, the model occasionally adds edges to the condition graph to achieve a lower total score, resulting in jumps in the number of edges.}
    \label{fig:charts}
\end{figure}

\section{Using different ML models}

A promising machine learning approach to determining Stanley coefficients is to apply a supervised learning model. This idea is attractive for several reasons:
\begin{enumerate}
    \item The input data consists of a sequence of \(n\) unit intervals on a line, and the corresponding unit interval graph can be encoded as an increasing sequence of $n$ integers between $0$ and $n-1$, also called the area sequence.
    \item The output is a series of integer coefficients. With the right model, saliency analysis can be performed to identify which features of the Unit Interval Orders (UIOs) influence a particular coefficient. 
\end{enumerate}

We experimented with various feed-forward neural networks as well as small BERT-type GPT models, utilizing encoder transformers with full attention. This method has two significant limitations: (1) The training set is limited. Computing Stanley coefficients for UIOs is computationally expensive, requiring substantial memory, and we are constrained to sequences of length around 10. (2) While feature analysis is possible, it doesn't provide a true explanation for the coefficients. 
However, recent advances in the mathematical interpretability of transformer models have prompted us to explore this avenue in our ongoing research.

\bibliographystyle{abbrv}
\bibliography{BercziKluver.bib}

\end{document}